\documentclass[11pt,reqno]{amsart}
\usepackage{amsmath,amscd,amsthm,a4wide,amssymb}

\allowdisplaybreaks

\theoremstyle{plain}
\newtheorem{thm}{Theorem}[section]
\newtheorem{lem}{Lemma}[section]

\newtheorem{prop}{Proposition}[section]
\theoremstyle{definition}
\newtheorem{rem}{Remark}[section]

\newtheorem{defi}{Definition}[section]

\newtheorem{conv}{Convention}[section]

\numberwithin{equation}{section}

\def\rhs{\operatorname{RHS}}
\def\lhs{\operatorname{LHS}}

\def\esup{\operatornamewithlimits{ess\,sup}}

\def\RHS{\operatorname{RHS}}
\def\LHS{\operatorname{LHS}}

\def\R{\mathbb R}
\def\Z{\mathbb Z}
\def\ap{\approx}

\def\qq{\qquad}

\def\a{\alpha}
\def\b{\beta}

\def\O{\Omega}

\def\la{\lambda}

\def\vp{\varphi}
\def\s{\sigma}
\def\i{\infty}
\def\t{\theta}
\def\I{(0,\i)}

\def\rw{\rightarrow}

\def\dn{\downarrow}

\def\uu{\uparrow\uparrow}
\def\dd{\downarrow\downarrow}
\def\ls{\lesssim}
\def\gs{\gtrsim}

\def\R{\mathbb R}

\def\M{\mathcal M}

\def\mp{{\mathcal M}}
\def\W{{\mathcal W}}

\begin{document}

\title[Some new iterated Hardy-type inequalities: The case $\t = 1$]{Some new iterated Hardy-type inequalities: The case $\t = 1$}

\author[A.~Gogatishvili, R.Ch.~Mustafayev and L.-E.~Persson]{A.~Gogatishvili, R.Ch.~Mustafayev and L.-E.~Persson}

\thanks{The research of the first author was partly supported by the grants
201/08/0383 and 13-14743S of the Grant Agency of the Czech Republic and RVO:
67985840.  The research of the first and second authors was partly
supported by the joint project between  Academy of Sciences of
Czech Republic and The Scientific and Technological Research
Council of Turkey}

\begin{abstract}
In this paper we characterize the validity of the Hardy-type
inequality
\begin{equation*}
\left\|\left\|\int_s^{\i}h(z)dz\right\|_{p,u,(0,t)}\right\|_{q,w,\I}\leq
c \,\|h\|_{1,v,\I}
\end{equation*}
where $0<p< \i$, $0<q\leq +\infty$, $u$, $w$ and $v$ are weight
functions on  $(0,\infty)$. It is pointed out that this
characterization can be used to obtain new characterizations for
the boundedness between weighted Lebesgue spaces  for Hardy-type
operators restricted to the cone of monotone functions and for the
generalized Stieltjes operator.
\end{abstract}


\maketitle


\section{Introduction}\label{in}

Throughout the paper we assume that $I : = (a,b)\subseteq (0,\i)$.
By $\mp (I)$ we denote the set of all measurable functions on $I$.
The symbol $\mp^+ (I)$ stands for the collection of all $f\in\mp
(I)$ which are non-negative on $I$, while $\mp^+ (I,;\dn)$ is used
to denote the subset of those functions which are non-increasing
on $I$. The family of all weight functions (also called just
weights) on $I$, that is, locally integrable non-negative
functions on $\I$, is denoted by $\W (I)$.

For $p\in (0,+\i]$ and $w\in \mp^+(I)$, we define the functional
$\|\cdot\|_{p,w,I}$ on $\mp (I)$ by
\begin{equation*}
\|f\|_{p,w,I} : = \left\{\begin{array}{cl}
                             \left(\int_I |f(x)|^p w(x)\,dx \right)^{1/p} & \qq\mbox{if}\qq p<+\i \\
                             \esup_{I} |f(x)|w(x) & \qq\mbox{if}\qq p=+\i.
                           \end{array}
\right.
\end{equation*}

If, in addition, $w\in \W(I)$, then the weighted Lebesgue space
$L^p(w,I)$ is given by
\begin{equation*}
L^p(w,I) = \{f\in \mp (I):\,\, \|f\|_{p,w,I} < +\i\}
\end{equation*}
and it is equipped with the quasi-norm $\|\cdot\|_{p,w,I}$.

When $w\equiv 1$ on $I$, we write simply $L^p(I)$ and
$\|\cdot\|_{p,I}$ instead of $L^p(w,I)$ and $\|\cdot\|_{p,w,I}$,
respectively.

Everywhere in the paper, $u$, $v$ and $w$ are weights. We denote
by
$$
U(t) : =\int_0^t u(s)ds,\qq V(t) : =\int_0^t v(s)ds \qq \mbox{for
every}\,\, t\in\I,
$$
and assume that $U(t) > 0$ for every $t \in \I$.

In this paper we characterize the validity of the inequality
\begin{equation}\label{mainn}
\left\|\left\|\int_s^{\i}h(z)dz\right\|_{p,u,(0,t)}\right\|_{q,w,\I}\leq
c \|h\|_{\t,v,\I}
\end{equation}
where $0<p<\i$, $0<q\leq +\infty$, $\t = 1$, $u$, $w$ and $v$ are
weight functions on  $(0,\infty)$. Note that inequality
\eqref{mainn} have been considered in the case $p=1$ in
\cite{gop2009} (see also \cite{g1}), where the result is presented
without proof, in the case $p=\i$ in \cite{gop} and in the case
$\t=1$ in \cite{gjop} and \cite{ss}, where the special type of
weight function $v$ was considered, and, recently, in \cite{gmp}
in the case $0<p<\i$, $0<q\leq +\infty$, $1 < \t \le \i$.

We pronounce that the characterization of the inequality
\eqref{mainn} is important because many inequalities for classical
operators  can be reduced to this form.  Just to illustrate this
important fact we give two applications in Section
\ref{Some.Applications} of the obtained results. Firstly, we
present some new characterizations of weighted Hardy-type
inequalities restricted to the cone of monotone functions (see
Theorems~\ref{main3} and \ref{main30000}). Secondly, we point out
boundedness results in weighted Lebesgue spaces concerning the
weighted Stieltjes's transform  (see
Theorems~\ref{Stieltjes.ineq.thm.1.1} and
\ref{Stieltjes.ineq.02}). Here we also need to prove some
reduction theorems of independent interest (see
Theorems~\ref{Reduction.thm.1.1},  \ref{Reduction.thm.1.2} and
\ref{Reduction.thm.1.3}).

Our approach is based on discretization and anti-discretization
methods developed in \cite{gp1}, \cite{gp2}, \cite{gjop} and
\cite{gmp}. Some basic facts concerning these methods and other
preliminaries are presented in Section \ref{pre}. In Section
\ref{mr} discretizations of the inequalities \eqref{mainn} are
given. Anti-discretization of the obtained conditions in Section
\ref{mr} and the main results (Theorems \ref{main1001},
\ref{main1002} and \ref{main2002}) are stated and proved in
Section \ref{Antidiscretization}. Finally,  the described
applications  can be found in Section \ref{Some.Applications}.

\section{Notations and Preliminaries}\label{pre}

Throughout the paper, we always denote by  $c$ or $C$ a positive
constant, which is independent of the main parameters but it may
vary from line to line. However a constant with subscript such as
$c_1$ does not change in different occurrences. By $a\lesssim b$,
($b\gtrsim a$) we mean that $a\leq \la b$, where $\la >0$ depends
only on inessential parameters. If $a\lesssim b$ and $b\lesssim
a$, we write $a\approx b$ and say that $a$ and $b$ are
equivalent. Throughout the paper we use the abbreviation $\LHS
(*)$ ($\RHS(*)$) for the left (right) hand side of the relation
$(*)$. By $\chi_Q$ we denote the characteristic function of a set
$Q$. Unless a special remark is made, the differential element
$dx$ is omitted when the integrals under consideration are the
Lebesgue integrals.

\begin{conv}\label{Notat.and.prelim.conv.1.1}
{\rm (i)} Throughout the paper we put $1 / (+\i) =0$, $(+\i) /
(+\i) = 0$, $1 / 0 = (+\i)$, $0/0 = 0$, $0 \cdot (\pm \i) = 0$,
$(+\i)^{\a} = +\i$ and $\a^0 = 1$ if $\a \in (0,+\i)$.

{\rm (ii)} If $p\in [1,+\i]$, we define $p'$ by $1/p + 1/p' = 1$.
Moreover, we put $p^* = \frac{p}{1-p}$ if $p\in (0,1)$ and $p^* =
+\i$ if $p=1$.

{\rm (iii)} If $I = (a,b) \subseteq \R$ and $g$ is a monotone
function on $I$, then by $g(a)$ and $g(b)$ we mean the limits
$\lim_{x\rw a+}g(x)$ and $\lim_{x\rw b-}g(x)$, respectively.
\end{conv}

In the paper we shall use the Lebesgue-Stieltjes integral. To this
end, we recall some basic facts.

Let $\vp$ be non-decreasing and finite function on the interval $I
: = (a,b)\subseteq \R$. We assign to $\vp$ the function $\la$
defined on subintervals of $I$ by
\begin{align}
\la ([\a,\b]) & = \vp (\b+) - \vp(\a-), \label{eq+-}\\
\la ([\a,\b)) & = \vp (\b-) - \vp(\a-), \label{eq--}\\
\la ((\a,\b]) & = \vp (\b+) - \vp(\a+), \label{eq++}\\
\la ((\a,\b)) & = \vp (\b-) - \vp(\a+). \label{eq-+}
\end{align}
The function $\la$ is a non-negative, additive and regular
function of intervals. Thus (cf. \cite{ru}, Chapter 10), it admits
a unique extension to a non-negative Borel measure $\la$ on $I$.

The formula \eqref{eq--} imply that
\begin{equation}\label{eq.alphabeta}
\int_{[\a,\b)} d\vp = \vp(\b-) - \vp(\a-).
\end{equation}

Note also that the associated Borel measure can be determined,
e.g., only by putting
$$
\la ([y,z]) = \vp(z+) - \vp(y-) \qq \mbox{for any}\qq [y,z]\subset
I
$$
(since the Borel subsets of $I$ can be generated by subintervals
$[y,z]\subset I$).

If $J\subseteq I$, then the Lebesgue-Stieltjes integral $\int_J
f\,d\vp$ is defined as $\int_J f\,d\la$. We shall also use the
Lebesgue-Stieltjes integral $\int_J f\,d\vp$ when $\vp$ is a
non-increasing and finite on the interval $I$. In such a case we
put
$$
\int_J f\,d\vp : = - \int_J f\,d(-\vp).
$$

We conclude this section by recalling an integration by parts
formula for Lebeshgue-Stieltjes integrals. For any non-decreasing
function $f$ and a continuous function $g$ on $\R$ the following
formula is valid for $-\i < \a < \b < \i$:
\begin{equation}\label{integration.by.parts}
\int_{[\a,\b)} f(t)\,d(g(t)) = f(\b-)g(\b) - f(\a-)g(\a) +
\int_{[\a,\b)}g(t)d(-f(t-)).
\end{equation}

\begin{rem}\label{Notat.and.prelim.rem.1.2}
Let $I = (a,b)\subseteq \R$. If $f\in C(I)$ and $\vp$ is a
non-decreasing, right continuous and finite function on $I$, then
it is possible to show that, for any $[y,z]\subset I$, the
Riemann-Stieltjes integral $\int_{[y,z]}f\,d\vp$ (written usually
as $\int_y^z f\,d\vp$) coincides with the Lebesgue-Stieltjes
integral $\int_{(y,z]}f\,d\vp$. In particular, if $f,\,g\in C(I)$
and $\vp$ is non-decreasing on $I$, then the Riemann-Stieltjes
integral $\int_{[y,z]}f\,d\vp$ coincides with the
Lebesgue-Stieltjes integral $\int_{(y,z]}f\,d\vp$ for any
$[y,z]\subset I$.
\end{rem}

Let us now recall some definitions and basic facts concerning
discretization and anti-discretization which can be found in
\cite{gp1}, \cite{gp2} and \cite{gjop}.
\begin{defi}\label{def.2.1}
Let $\{a_k\}$ be a sequence of positive real numbers. We say that
$\{a_k\}$ is geometrically increasing or geometrically decreasing
and write $a_k\uu$ or $a_k\dd$ when
$$
\inf_{k\in\Z}\frac{a_{k+1}}{a_k}>1 ~~\mbox{or}
~~\sup_{k\in\Z}\frac{a_{k+1}}{a_k}<1,
$$
respectively.
\end{defi}

\begin{defi}\label{def.2.2}
Let $U$ be a continuous strictly increasing function on $[0,\i)$
such that $U(0)=0$ and $\lim\limits_{t\rightarrow\i}U(t)=\i$. Then
we say that $U$ is admissible.
\end{defi}

Let $U$ be an admissible function. We say that a function $\vp$ is
$U$-quasiconcave if $\vp$ is equivalent to an increasing function on
$(0,\i)$ and $\frac{\vp}{U}$ is equivalent to a decreasing function
on $(0,\i)$. We say that a $U$-quasiconcave function $\vp$ is
non-degenerate if
$$
\lim_{t\rightarrow
0+}\vp(t)=\lim_{t\rightarrow\i}\frac{1}{\vp(t)}=\lim_{t\rightarrow\i}\frac{\vp(t)}{U(t)}=\lim_{t\rightarrow
0+}\frac{U(t)}{\vp(t)}=0.
$$
The family of non-degenerate $U$-quasiconcave functions will be
denoted by $\O_U$. We say that $\vp$ is quasiconcave when
$\vp\in\O_U$ with $U(t)=t$. A quasiconcave function is equivalent to
a concave function. Such functions are very important in various
parts of analysis. Let us just mention that e.g. the Hardy operator
$Hf(x) = \int_0^x f(t)dt$ of a decreasing function, the Peetre
$K$-functional in interpolation theory and the fundamental function
$\|\chi_E\|_X$, $X$ is a rearrangement invariant space, all are
quasiconcave.

\begin{defi}\label{def.2.3}
Assume that $U$ is admissible and $\vp\in\O_U$. We say that
$\{x_k\}_{k\in\Z}$ is a discretizing sequence for $\vp$ with respect
to $U$ if

(i) $x_0=1$ and $U(x_k)\uu$;

(ii) $\vp(x_k)\uu$ and $\frac{\vp(x_k)}{U(x_k)}\dd$;

(iii) there is a decomposition $\Z=\Z_1\cup\Z_2$ such that
$\Z_1\cap\Z_2=\emptyset$ and for every $t\in [x_k,x_{k+1}]$
$$
\vp(x_k)\thickapprox \vp(t) ~~\mbox{if} ~~ k\in\Z_1,
$$
$$
\frac{\vp(x_k)}{U(x_k)}\thickapprox \frac{\vp(t)}{U(t)} ~~\mbox{if}
~~ k\in\Z_2.
$$
\end{defi}

Let us recall (\cite{gp1}, Lemma 2.7) that if $\vp\in\O_U$, then
there always exists a discretizing sequence for $\vp$ with respect
to $U$.

\begin{defi}\label{def.2.4}
Let $U$ be an admissible function and let $\nu$ be a non-negative
Borel measure on $[0,\i)$. We say that the function $\vp$ defined by
$$
\vp(t)=U(t)\int_{[0,\i)}\frac{d\nu(s)}{U(s)+U(t)},~~t\in(0,\i),
$$
is the fundamental function of the measure $\nu$ with respect to
$U$. We will also say that $\nu$ is a representation measure of
$\vp$ with rspect to $U$.

We say that $\nu$ is non-degenerate with respect to $U$ if the
following conditions are satisfied for every $t\in\I$:
$$
\int_{[0,\i)}\frac{d\nu(s)}{U(s)+U(t)}<\i,~t\in\I~ \mbox{and}
~\int_{[0,1]}\frac{d\nu(s)}{U(s)}=\int_{[1,\i)}d\nu(s)=\i.
$$
\end{defi}

We recall from Remark 2.10 of \cite{gp1} that
$$
\vp(t)\ap
\int_{[0,t]}d\nu(s)+U(t)\int_{[t,\i)}U(s)^{-1}d\nu(s),~~t\in(0,\i).
$$

\begin{lem}\label{cor.2.0}{\rm(\cite{gp2}, Lemma 1.5).}
Let $p \in \I$. Let $u$, $w$ be weights and let $\vp$ be defined
by
\begin{equation}\label{eq.2.0}
\vp(t)=\esup_{s\in(0,t)}{U(s)}^{\frac{1}{p}}
\esup_{\tau\in(s,\i)}\frac{w(\tau)}{{U(\tau)}^{\frac{1}{p}}},~~t\in\I.
\end{equation}
Then $\vp$ is the least $U^{\frac{1}{p}}$-quasiconcave majorant of
$w$, and
\begin{equation*}
\begin{split}
\sup_{t\in\I}\vp(t)\left(\frac{1}{U(t)}\int_0^t\left(\int_s^{\i}h(z)dz\right)^pu(s)ds\right)
^{\frac{1}{p}}&\\
&\hspace{-5cm}=\esup_{t\in\I}w(t)\left(\frac{1}{U(t)}\int_0^t\left(\int_s^{\i}h(z)dz\right)^pu(s)ds\right)
^{\frac{1}{p}}
\end{split}
\end{equation*}
for any non-negative measurable $h$ on $\I$. Further, for $t\in\I$
\begin{gather*}
\vp(t)=\esup_{\tau\in\I}w(\tau)\min\left\{1,\left(\frac{U(t)}{U(\tau)}\right)^{\frac{1}{p}}\right\}=
{U(t)}^{\frac{1}{p}}\esup_{s\in(t,\i)}\frac{1}{{U(s)}^{\frac{1}{p}}}\esup_{\tau\in(0,s)}w(\tau),\\
\vp(t)\ap\esup_{s\in\I}w(s)\left(\frac{U(t)}{U(s)+U(t)}\right)^{\frac{1}{p}}.
\end{gather*}
\end{lem}

\begin{thm}\label{thm.2.1}{\rm(\cite{gp1}, Theorem 2.11).}
Let $p,\,q,\,r\in\I$. Assume that $U$ is an admissible function,
$\nu$ is a non-negative non-degenerate Borel measure on $[0,\i)$,
and $\vp$ is the fundamental function of $\nu$ with respect to $U^q$
and $\s\in\O_{U^p}$. If $\{x_k\}$ is a discretizing sequence for
$\vp$ with respect to $U^q$, then
$$
\int_{[0,\i)}\frac{\vp(t)^{\frac{r}{q}-1}}{\s(t)^{\frac{r}{p}}}d\nu(t)\ap
\sum_{k\in\Z}\frac{\vp(x_k)^{\frac{r}{q}}}{\s(x_k)^{\frac{r}{p}}}.
$$
\end{thm}

\begin{lem}\label{lem.2.1}{\rm(\cite{gp1}, Corollary 2.13).}
Let $q\in\I$. Assume that $U$ is an admissible function, $f\in\O_U$,
$\nu$ is a non-negative non-degenerate Borel measure on $[0,\i)$ and
$\vp$ is the fundamental function of $\nu$ with respect to $U^q$. If
$\{x_k\}$ is a discretizing sequence for $\vp$ with respect to
$U^q$, then
$$
\left(\int_{[0,\i)}\left(\frac{f(t)}{U(t)}\right)^qd\nu(t)\right)^{\frac{1}{q}}\ap
\left(\sum_{k\in\Z}\left(\frac{f(x_k)}{U(x_k)}\right)^q
\vp(x_k)\right)^{\frac{1}{q}}.
$$
\end{lem}

\begin{lem}\label{lem.2.2}{\rm(\cite{gp1}, Lemma 3.5).}
Let $p,\,q\,\in\I$. Assume that $U$ is an admissible function,
$\vp\in\O_{U^q}$ and $g\in \O_{U^p}$. If $\{x_k\}$ is a
discretizing sequence for $\vp$ with respect to $U^q$, then
$$
\sup_{t\in\I}\frac{\vp(t)^{\frac{1}{q}}}{g(t)^{\frac{1}{p}}}\ap\sup_{k\in\Z}\frac{\vp(x_k)^{\frac{1}{q}}}{g(x_k)^{\frac{1}{p}}}.
$$
\end{lem}

%

We shall use some Hardy-type inequalities in this paper. Denote by
$$
\underline{v}(a,b) : = \esup\limits_{s \in I} {v(s)}^{-1},
$$
\begin{equation}\label{Bk}
B(a,b) : = \sup_{h\in \mp^+(I)} \left\|\int_s^b
h(z)dz\right\|_{p,u,I} / \left\|h\right\|_{1,v,I}.
\end{equation}
\begin{lem}\label{lem.2.800}
We have the following Hardy-type inequalities:

{\rm (a)} Let $1\leq p < \i$. Then the inequality
\begin{equation}\label{eq.2.200}
\left\|\int_s^b h(z)dz\right\|_{p,u,I}\leq c
\left\|h\right\|_{1,v,I}
\end{equation}
holds for all $h \in \mp^+(I)$ if and only if
$$
\sup_{t \in I}\left(\int_a^t u(z)dz\right)^{\frac{1}{p}}
\underline{v}(t,b) < \i,
$$
and the best constant $c = B(a,b)$ in \eqref{eq.2.200} satisfies
\begin{equation}\label{eq.2.20001}
B(a,b) \ap \sup_{t\in I}\left(\int_a^t u(z)dz\right)^{\frac{1}{p}}
\underline{v}(t,b).
\end{equation}

{\rm (b)} Let $0<p<1$. Then inequality \eqref{eq.2.200} holds for
all $h \in \mp^+(I)$ if and only if
$$
\left(\int_a^b \left(\int_a^t
u(z)dz\right)^{p^*}u(t)\underline{v}(t,b)^{p^*}\,dt
\right)^{\frac{1}{p^*}} < \i,
$$
and
\begin{align*}
B(a,b) \ap \left(\int_a^b \left(\int_a^t
u(z)dz\right)^{p^*}u(t)\underline{v}(t,b)^{p^*}\,dt
\right)^{\frac{1}{p^*}}.
\end{align*}
\end{lem}

These well-known results can be found in Maz'ya and Rozin \cite{m},
Sinnamon \cite{s}, Sinnamon and Stepanov \cite{ss} (cf. also
\cite{ok} and \cite{kp}).

We shall also use the following fact (cf. \cite{f}, p. 188):
\begin{equation}\label{Ck}
C(a,b) : = \sup_{h \in \M^+(I)} \left\|h\right\|_{1,I} /
\|h\|_{1,v,I} \ap \underline{v}(a,b).
\end{equation}

Finally, if $q\in (0,+\infty]$ and $\{w_k\}=\{w_k\}_{k\in \Z}$ is
a sequence of positive numbers, we denote by $\ell^q(\{w_k\},\Z)$
the following  discrete analogue of a weighted Lebesgue space: if
$ 0<q<+\infty$, then
\begin{align*}
&\ell^q(\{w_k\},\Z) =\big\{ \{a_k\}_{k\in\Z} :\quad
\|a_k\|_{\ell^q(\{w_k\},\Z)}
:=\big(\sum_{k\in \Z}|a_kw_k|^q\big)^{\frac 1q}<+\infty \big\}\\
\intertext{and} &\ell^\infty(\{w_k\},\Z) =\big\{ \{a_k\}_{k\in\Z}
:\quad
\|a_k\|_{\ell^\infty(\{w_k\},\Z)}:=\sup_{k\in\Z}|a_kw_k|<+\infty
\big\}.
\end{align*}
If $w_k=1$ for all $k\in\Z$, we write simply $\ell^q(\Z)$ instead
of $\ell^q(\{w_k\},\Z)$.

We quote some known results. Proofs  can be found in \cite{Le1} and
\cite{Le2}.
\begin{lem}\label{lem.2.3}
Let $q\in (0,+\infty ]$. If $\{ \tau_k\} _{k\in\Z}$ is a
geometrically decreasing sequence, then
\begin{equation*}
\left\| \tau _k \sum_{m \le k} a_m\right\| _{\ell^q(\Z)} \approx \|
\tau _k a_k\| _{\ell^q(\Z)}
\end{equation*}
and
\begin{equation*}
\left\| \tau _k \sup _{m\leq k}a_m \right\| _{\ell^q(\Z)} \approx
\|\tau _ka_k\| _{\ell^q(\Z)}
\end{equation*}
for all non-negative sequences $\{a_k\}_{k\in\Z}$.

Let $\{ \sigma_k\} _{k\in\Z}$ be a geometrically increasing
sequence. Then
\begin{equation*}
\left\| \sigma _k \sum_{m \ge k}a_m\right\| _{\ell^q(\Z)} \approx \|
\sigma _ka_k\| _{\ell^q(\Z)}
\end{equation*}
and
\begin{equation*}
\left\| \sigma _k \sup _{m\geq k}a_m \right\| _{\ell^q(\Z)} \approx
\|\sigma _ka_k\| _{\ell^q(\Z)}
\end{equation*}
for all non-negative sequences $\{ a_k\} _{k\in\Z}$.
\end{lem}

We shall use the following inequality, which is a simple consequence
of the discrete H\"{o}lder inequality:
\begin{equation}\label{discrete.Hold.}
\|\{ a_k b_k \}\|_{\ell^q (\Z)} \le \|\{ a_k \}\|_{\ell^{\rho}
(\Z)} \|\{ b_k \}\|_{\ell^p (\Z)},
\end{equation}
where $\frac{1}{\rho} = \left( \frac{1}{q} -
\frac{1}{p}\right)_+$. \footnote{For any $a\in\R$ denote by $a_+ =
a$ when $a>0$ and $a_+ = 0$ when $a \le 0$.}

Given two (quasi-)Banach spaces $X$ and $Y$, we write $X
\hookrightarrow Y$ if $X \subset Y$ and if the natural embedding
of $X$ in $Y$ is continuous.

The following two lemmas are discrete version of the classical
Landau resonance theorems. Proofs can be found, for example, in
\cite{gp1}.
\begin{prop}\label{prop.2.1}{\rm(\cite{gp1}, Proposition 4.1).}
Let $0 < p,\, q \le \i$, and let $\{v_k\}_{k\in\Z}$ and
$\{w_k\}_{k\in\Z}$ be two sequences of positive numbers. Assume
that
\begin{equation}\label{eq31-4651}
\ell^p (\{v_k\},\Z) \hookrightarrow \ell^q (\{w_k\},\Z).
\end{equation}

{\rm (i)} If $0 < p \le q \le \i$, then
\begin{equation*}\label{eq31-46519009}
\|\{w_k v_k^{-1}\}\|_{\ell^{\i}(\Z)} \le C,
\end{equation*}
where $C$ stands for the norm of the inequality \eqref{eq31-4651}.

{\rm (ii)} If $0 < q \le p \le \i$, then
\begin{equation*}\label{eq31-407676519009}
\|\{w_k v_k^{-1}\}\|_{\ell^{r}(\Z)} \le C,
\end{equation*}
where $1/r : = 1/q - 1/p$ and $C$ stands for the norm of the
inequality \eqref{eq31-4651}.
\end{prop}

\section{Discretization of Inequalities}\label{mr}

In this section we discretize the inequalities
\begin{equation}\label{eq.4.1}
\begin{split}
\left(\int_0^{\i}\left(\frac{1}{U(t)}\int_0^t\left(\int_s^{\i}h(z)dz\right)^pu(s)ds\right)
^{\frac{q}{p}}w(t)dt\right)^{\frac{1}{q}}\leq c \int_0^{\i}
h(z)v(z)\,dz,
\end{split}
\end{equation}
and
\begin{equation}\label{eq.4.2}
\sup_{t\in\I}w(t)\left(\frac{1}{U(t)}\int_0^t\left(\int_s^{\i}h(z)dz\right)^pu(s)ds\right)
^{\frac{1}{p}} \leq c \int_0^{\i} h(z)v(z)\,dz.
\end{equation}
We start with inequality \eqref{eq.4.1}. At first we do the
following remark.
\begin{rem}
Let $\vp$ be the fundamental function of the measure $w(t)dt$ with
respect to $U^{\frac{q}{p}}$, that is,
\begin{equation}\label{eq.4.3}
\vp(x): =\int_0^{\i}{\mathcal U}(x,s)^{\frac{q}{p}}w(s)ds
\qquad\mbox{for all} \qquad x\in \I,
\end{equation}
where
$$
{\mathcal U}(x,t): = \frac{U(x)}{U(t)+U(x)}.
$$
Assume that  $w(t)dt$ is non-degenerate with respect to
$U^{\frac{q}{p}}$. Then $\vp\in\O_{U^{\frac{q}{p}}}$, and therefore
there exists a discretizing sequence for $\vp$ with respect to
$U^{\frac{q}{p}}$. Let $\{x_k\}$ be one such sequence. Then
$\vp(x_k)\uu$ and $\vp(x_k)U^{-\frac{q}{p}}\dd$. Furthermore, there
is a decomposition $\Z=\Z_1\cup\Z_2$, $\Z_1\cap\Z_2=\emptyset$ such
that for every $k\in\Z_1$ and $t\in[x_k,x_{k+1}]$, $\vp (x_k)
\ap\vp(t)$ and for every $k\in\Z_2$ and $t\in[x_k,x_{k+1}]$,
$\vp(x_k){U(x_k)}^{-\frac{q}{p}}\ap\vp(t){U(t)}^{-\frac{q}{p}}$.
\end{rem}
Next, we state a necessary lemma which is also of independent
interest.
\begin{lem}\label{lem14357861}
Let $0<q<\i$, $0<p<\i$, $1 / \rho = (1 / q - 1)_+$, and let
$u,\,v,\,w$ be weights. Assume that $u$ is such that $U$ is
admissible and the measure $w(t)dt$ is non-degenerate with respect
to $U^{\frac{q}{p}}$. Let $\{x_k\}$ be any discretizing sequence
for $\vp$ defined by \eqref{eq.4.3}. Then inequality
\eqref{eq.4.1} holds  for every $h \in \mp^+(0,\i)$ if and only if
\begin{equation}\label{A}
A : = \left\| \left\{
\frac{\vp(x_k)^{\frac{1}{q}}}{U(x_k)^{\frac{1}{p}}}B(x_{k-1},x_k)\right\}
\right\|_{\ell^{\rho}(\Z)} + \left\| \left\{
\vp(x_k)^{\frac{1}{q}}C(x_k,x_{k+1})\right\}
\right\|_{\ell^{\rho}(\Z)} < \i,
\end{equation}
and the best constant in inequality \eqref{eq.4.1} satisfies
$$
c\ap A.
$$
\end{lem}
\begin{proof}
By using Lemma \ref{lem.2.1} with
$$
d\nu(t)=w(t)dt \qquad\mbox{and}\qquad
f(t)=\int_0^t\left(\int_s^{\i}h(z)dz\right)^pu(s)ds
$$
we get that
\begin{equation*}\label{eq134-6813=0968}
\lhs \eqref{eq.4.1} \ap
\left\|\left\{\left\|\int_s^{\i}h(z)dz\right\|_{p,u,(0,x_k)}\frac{\vp(x_k)^{\frac{1}{q}}}{U(x_k)^{\frac{1}{p}}}\right\}\right\|_{\ell^q(\Z)}.
\end{equation*}
Moreover, by using Lemma \ref{lem.2.3},
\begin{align*}
\lhs \eqref{eq.4.1}  \ap &
\left\|\left\{\left\|\int_s^{\i}h(z)dz\right\|_{p,u,I_k}\frac{\vp(x_k)^{\frac{1}{q}}}{U(x_k)^{\frac{1}{p}}}\right\}\right\|_{\ell^q(\Z)}
\\
\ap &
\left\|\left\{\left\|\int_s^{x_k}h(z)dz+\int_{x_k}^{\i}h(z)dz\right\|_{p,u,I_k}\frac{\vp(x_k)^{\frac{1}{q}}}{U(x_k)^{\frac{1}{p}}}\right\}\right\|_{\ell^q(\Z)}
\\
\ap &
\left\|\left\{\left\|\int_s^{x_k}h(z)dz\right\|_{p,u,I_k}\frac{\vp(x_k)^{\frac{1}{q}}}{U(x_k)^{\frac{1}{p}}}\right\}\right\|_{\ell^q(\Z)}
\\
&+\left\|\left\{\left\|\int_{x_k}^{\i}h(z)dz\right\|_{p,u,I_k}\frac{\vp(x_k)^{\frac{1}{q}}}{U(x_k)^{\frac{1}{p}}}\right\}\right\|_{\ell^q(\Z)}
\\
\ap
&\left\|\left\{\left\|\int_s^{x_k}h(z)dz\right\|_{p,u,I_k}\frac{\vp(x_k)^{\frac{1}{q}}}{U(x_k)^{\frac{1}{p}}}\right\}\right\|_{\ell^q(\Z)}\\
&+\left\|\left\{\int_{x_k}^{\i}h(z)dz \left\|1\right\|_{p,u,I_k}
\frac{\vp(x_k)^{\frac{1}{q}}}{U(x_k)^{\frac{1}{p}}}\right\}\right\|_{\ell^q(\Z)},
\end{align*}
where $I_k : = (x_{k-1},x_k)$, $k \in \Z$. By now using the fact
that
$$
\|1\|_{p,u,I_k}
=\int_{x_{k-1}}^{x_k}u(s)ds=U(x_k)-U(x_{k-1})\ap U(x_k)
$$
we find that
\begin{align*}
\lhs \eqref{eq.4.1}  \ap &
\left\|\left\{\left\|\int_s^{x_k}h(z)dz\right\|_{p,u,I_k}\frac{\vp(x_k)^{\frac{1}{q}}}{U(x_k)^{\frac{1}{p}}}\right\}\right\|_{\ell^q(\Z)}\\
&+\left\|\left\{\vp(x_k)^{\frac{1}{q}}\int_{x_k}^{\i}h(z)dz
\right\}\right\|_{\ell^q(\Z)}.
\end{align*}
Consequently, by using Lemma \ref{lem.2.3} on the second term,
\begin{gather}
\lhs \eqref{eq.4.1}  \ap
\left\|\left\{\left\|\int_s^{x_k}h(z)dz\right\|_{p,u,I_k}\frac{\vp(x_k)^{\frac{1}{q}}}{U(x_k)^{\frac{1}{p}}}\right\}\right\|_{\ell^q(\Z)} \notag\\
+\left\|\left\{\vp(x_k)^{\frac{1}{q}}\int_{x_k}^{x_{k+1}}h(z)dz
\right\}\right\|_{\ell^q(\Z)}:=I+II. \label{eq.4.4}
\end{gather}
To find a sufficient condition for the validity of inequality
\eqref{eq.4.1}, we apply to $I$ locally (that is, for any $k\in \Z$)
the Hardy-type inequality
\begin{equation}\label{localHardy}
\left\|\int_s^{x_k}h(z)dz\right\|_{p,u,I_k} \le B(x_{k-1},x_k)
\|h\|_{1,v,I_k},\qquad h\in \M^+(I_k).
\end{equation}
Thus, in view of inequality \eqref{discrete.Hold.}, we have that
\begin{align}\label{I}
I  \le \left\|\left\{B(x_{k-1},x_k)
\frac{\vp(x_k)^{\frac{1}{q}}}{U(x_k)^{\frac{1}{p}}}\|h\|_{1,v,I_k}\right\}\right\|_{\ell^q(\Z)}
& \notag\\
& \hspace{-4cm} \le \left\|\left\{B(x_{k-1},x_k)
\frac{\vp(x_k)^{\frac{1}{q}}}{U(x_k)^{\frac{1}{p}}}\right\}\right\|_{\ell^{\rho}(\Z)}
\left\|\{\|h\|_{1,v,I_k}\}\right\|_{\ell^1(\Z)} \nonumber\\
&  \hspace{-4cm} = \left\|\left\{B(x_{k-1},x_k)
\frac{\vp(x_k)^{\frac{1}{q}}}{U(x_k)^{\frac{1}{p}}}\right\}\right\|_{\ell^{\rho}(\Z)}
\|h\|_{1,v,\I}.
\end{align}
For $II$, by inequalities \eqref{Ck} and \eqref{discrete.Hold.},
we get that
\begin{align}
II =
\left\|\left\{\vp(x_k)^{\frac{1}{q}}\int_{x_k}^{x_{k+1}}h(z)dz
\right\}\right\|_{\ell^q(\Z)}\nonumber & \\
& \hspace{-3cm} \leq
\left\|\left\{\vp(x_k)^{\frac{1}{q}}C(x_k,x_{k+1})\|h\|_{1,v,I_{k+1}}
\right\}\right\|_{\ell^q(\Z)} \nonumber\\
& \hspace{-3cm} \leq
\left\|\left\{\vp(x_k)^{\frac{1}{q}}C(x_k,x_{k+1})\right\}
\right\|_{\ell^{\rho}(\Z)}\left\|\left\{\|h\|_{1,v,I_{k+1}}
\right\}\right\|_{\ell^1(\Z)}\nonumber\\
& \hspace{-3cm} =
\left\|\left\{\vp(x_k)^{\frac{1}{q}}C(x_k,x_{k+1})\right\}
\right\|_{\ell^{\rho}(\Z)}\|h\|_{1,v,\I}. \label{II}
\end{align}
Combining \eqref{I} and \eqref{II}, in view of \eqref{eq.4.4}, we
obtain that
\begin{align}
\lhs \eqref{eq.4.1} & \nonumber\\
&\hspace{-1.5cm}\ls \left(\left\|\left\{ B(x_{k-1},x_k)
\frac{\vp(x_k)^{\frac{1}{q}}}{U(x_k)^{\frac{1}{p}}}\right\}\right\|_{\ell^{\rho}(\Z)}
+ \left\|\left\{\vp(x_k)^{\frac{1}{q}}C(x_k,x_{k+1})\right\}
\right\|_{\ell^{\rho}(\Z)}\right)\rhs \eqref{eq.4.1}.
\end{align}
Consequently, \eqref{eq.4.1} holds provided that $A<\i$ and $c\le
A$.

Next we prove that condition \eqref{A} is also necessary for the
validity of inequality \eqref{eq.4.1}. Assume that inequality
\eqref{eq.4.1} holds with $c < \i$. By \eqref{Bk}, there are $h_k
\in \M^+(I_k)$, $k\in \Z$, such that
\begin{equation}\label{hknorm}
\|h_k\|_{1,v,I_k} = 1
\end{equation}
and
\begin{equation}\label{hk saturates}
\frac{1}{2} B(x_{k-1},x_k)\le
\left\|\int_s^{x_k}h_k(z)dz\right\|_{p,u,I_k} \qquad \mbox{for
all}\qquad k\in\Z.
\end{equation}
Define $g_k$, $k\in \Z$, as the extension of $h_k$ by $0$ to the
whole interval $\I$ and put
\begin{equation}\label{g}
g = \sum_{k\in \Z} a_k g_k,
\end{equation}
where $\{a_k\}_{k\in\Z}$ is any sequence of positive numbers. We
obtain that
\begin{align}
\lhs \eqref{eq.4.1} & \gs
\left\|\left\{\left\|\int_s^{x_k}\sum_{m\in \Z} a_m
g_m\right\|_{p,u,I_k}\frac{\vp(x_k)^{\frac{1}{q}}}{U(x_k)^{\frac{1}{p}}}\right\}\right\|_{\ell^q(\Z)}
\nonumber\\
& \gs \left\|\left\{a_k
B(x_{k-1},x_k)\frac{\vp(x_k)^{\frac{1}{q}}}{U(x_k)^{\frac{1}{p}}}\right\}\right\|_{\ell^q(\Z)}.
\label{nec1}
\end{align}
Moreover,
\begin{align}\label{nec2}
\rhs \eqref{eq.4.1}  = c\left\|\sum_{m\in \Z} a_m
g_m\right\|_{1,v,\I} =
c\left\|\left\{a_k\right\}\right\|_{\ell^1(\Z)}.
\end{align}
Therefore, by \eqref{eq.4.1}, \eqref{nec1} and \eqref{nec2}, we
arrive at
\begin{equation}
\left\|\left\{a_k
B(x_{k-1},x_k)\frac{\vp(x_k)^{\frac{1}{q}}}{U(x_k)^{\frac{1}{p}}}\right\}\right\|_{\ell^q(\Z)}
\ls c \left\|\left\{a_k\right\}\right\|_{\ell^1(\Z)},
\end{equation}
and Proposition \ref{prop.2.1} implies that
\begin{equation}\label{eq25-609820586y}
\left\| \left\{
\frac{\vp(x_k)^{\frac{1}{q}}}{U(x_k)^{\frac{1}{p}}}B(x_{k-1},x_k)\right\}
\right\|_{\ell^{\rho}(\Z)}  < c.
\end{equation}

On the other hand, there are $\psi_k \in \M^+(I_k)$, $k\in \Z$,
such that
\begin{equation}\label{eq13487561134551}
\|\psi_k\|_{1,v,I_k} = 1
\end{equation}
and
\begin{equation}\label{eq1349601=41345}
\|\psi_k\|_{1,I_{k+1}} \ge  \frac{1}{2} C(x_k,x_{k+1}) \qquad
\mbox{for all}\qquad k\in\Z.
\end{equation}
Define $f_k$, $k\in \Z$, as the extension of $\psi_k$ by $0$ to
the whole interval $\I$ and put
\begin{equation}\label{f}
f = \sum_{k\in \Z} b_k f_k,
\end{equation}
where $\{b_k\}_{k\in\Z}$ is any sequence of positive numbers. We
obtain that
\begin{align*}
\lhs \eqref{eq.4.1} & \ge
\left\|\left\{\vp(x_k)^{\frac{1}{q}}\int_{x_k}^{x_{k+1}}\sum_{m\in
\Z} b_m f_m \right\}\right\|_{\ell^q(\Z)}
\nonumber\\
& \gs \left\|\left\{b_k
\vp(x_k)^{\frac{1}{q}}C(x_k,x_{k+1})\right\}\right\|_{\ell^q(\Z)}.
\end{align*}
Note that
\begin{align*}
\rhs \eqref{eq.4.1}  = c \left\|\sum_{m\in \Z} b_m
f_m\right\|_{1,v,\I} = c
\left\|\left\{b_k\right\}\right\|_{\ell^1(\Z)}.
\end{align*}
Then, by \eqref{eq.4.1} and previous two inequalities, we have
that
$$
\left\|\left\{b_k
\vp(x_k)^{\frac{1}{q}}C(x_k,x_{k+1})\right\}\right\|_{\ell^q(\Z)}
\ls c \left\|\left\{b_k\right\}\right\|_{\ell^1(\Z)}.
$$
Proposition \ref{prop.2.1} implies that
\begin{equation}\label{eq205689728967}
\left\|\left\{
\vp(x_k)^{\frac{1}{q}}C(x_k,x_{k+1})\right\}\right\|_{\ell^{\rho}(\Z)}
< c.
\end{equation}
Inequalities \eqref{eq25-609820586y} and \eqref{eq205689728967}
prove that $A \ls c$.
\end{proof}

Before we proceed to inequality \eqref{eq.4.2} we make the following
remark.
\begin{rem}
Suppose that $\vp (x) < \i$ for all $x \in (0,\i)$, where $\vp$ is
defined by \eqref{eq.2.0}. Let $\vp$ be non-degenerate with
respect to $U^{\frac{1}{p}}$. Then, by Lemma \ref{cor.2.0},
$\vp\in\O_{U^{\frac{1}{p}}}$, and therefore there exists a
discretizing sequence for $\vp$ with respect to $U^{\frac{1}{p}}$.
Let $\{x_k\}$ be one such sequence. Then $\vp(x_k)\uu$ and
$\vp(x_k)U^{-\frac{1}{p}}\dd$. Furthermore, there is a
decomposition $\Z=\Z_1\cup\Z_2$, $\Z_1\cap\Z_2=\emptyset$ such
that for every $k\in\Z_1$ and $t\in[x_k,x_{k+1}]$, $\vp (x_k)
\ap\vp(t)$ and for every $k\in\Z_2$ and $t\in[x_k,x_{k+1}]$,
$\vp(x_k){U(x_k)}^{-\frac{1}{p}}\ap\vp(t){U(t)}^{-\frac{1}{p}}$.
\end{rem}
The following lemma is proved analogously, and for the sake of
completeness we give the full proof.
\begin{lem}\label{lem1435786143515}
Let $0<p<\i$ and let $u,\,v,\,w$ be weights. Assume that $u$ are
such that $U^{\frac{1}{p}}$ is admissible. Let $\vp$, defined by
\eqref{eq.2.0}, be non-degenerate with respect to
$U^{\frac{1}{p}}$. Let $\{x_k\}$ be any discretizing sequence for
$\vp$. Then inequality \eqref{eq.4.2} holds  for every $h \in
\mp^+(0,\i)$ if and only if
\begin{equation}\label{D}
D : = \left\| \left\{
\frac{\vp(x_k)}{U(x_k)^{\frac{1}{p}}}B(x_{k-1},x_k)\right\}
\right\|_{\ell^{\i}(\Z)} + \left\| \left\{
\vp(x_k)C(x_k,x_{k+1})\right\} \right\|_{\ell^{\i}(\Z)} < \i,
\end{equation}
and the best constant in inequality \eqref{eq.4.2} satisfies $c
\ap D$.
\end{lem}
\begin{proof}
Using Lemma \ref{cor.2.0}, Lemma \ref{lem.2.2}, Lemma \ref{lem.2.3},
we obtain for the left-hand side of \eqref{eq.4.2} that
\begin{align}
\lhs \eqref{eq.4.2} = &
\sup_{t\in\I}\frac{\vp(t)}{{U(t)}^{\frac{1}{p}}}\left\|\int_s^{\i}h(z)dz\right\|_{p,u,(0,t)}\nonumber \\
\ap &
\left\|\left\{\frac{\vp(x_k)}{{U(x_k)}^{\frac{1}{p}}}\left\|\int_s^{\i}h(z)dz\right\|_{p,u,(0,x_k)}\right\}\right\|_{\ell^{\i}(\Z)}\nonumber\\
\ap &
\left\|\left\{\frac{\vp(x_k)}{{U(x_k)}^{\frac{1}{p}}}\left\|\int_s^{\i}h(z)dz\right\|_{p,u,I_k}\right\}\right\|_{\ell^{\i}(\Z)}\nonumber\\
\ap &
\left\|\left\{\frac{\vp(x_k)}{{U(x_k)}^{\frac{1}{p}}}\left\|\int_s^{x_k}h(z)dz\right\|_{p,u,I_k}\right\}\right\|_{\ell^{\i}(\Z)}\nonumber\\
&+\left\|\left\{\vp(x_k)\int_{x_k}^{x_{k+1}}h(z)dz\right\}\right\|_{\ell^{\i}(\Z)}
: =III+IV. \label{eq.4.12}
\end{align}
To find a sufficient condition for the validity of inequality
\eqref{eq.4.2}, we apply to $III$ locally the Hardy-type inequality
\eqref{localHardy}. Thus
\begin{align}
III  \le \left\|\left\{B(x_{k-1},x_k)
\frac{\vp(x_k)}{U(x_k)^{\frac{1}{p}}}\|h\|_{1,v,I_k}\right\}\right\|_{\ell^{\i}(\Z)}
& \nonumber\\
& \hspace{-4cm} \le \left\|\left\{B(x_{k-1},x_k)
\frac{\vp(x_k)}{U(x_k)^{\frac{1}{p}}}\right\}\right\|_{\ell^{\i}(\Z)}
\left\|\{\|h\|_{1,v,I_k}\}\right\|_{\ell^1(\Z)} \nonumber\\
&  \hspace{-4cm} = \left\|\left\{B(x_{k-1},x_k)
\frac{\vp(x_k)}{U(x_k)^{\frac{1}{p}}}\right\}\right\|_{\ell^{\i}(\Z)}
\|h\|_{1,v,\I}. \label{III}
\end{align}
For $IV$ we have that
\begin{align}
IV =\left\|\left\{\vp(x_k)\int_{x_k}^{x_{k+1}}h(z)dz
\right\}\right\|_{\ell^{\i}(\Z)}\nonumber & \\
& \hspace{-4cm} \leq
\left\|\left\{\vp(x_k)C(x_k,x_{k+1})\|h\|_{1,v,I_{k+1}}
\right\}\right\|_{\ell^{\i}(\Z)} \nonumber\\
& \hspace{-4cm} \leq \left\|\left\{\vp(x_k)C(x_k,x_{k+1})\right\}
\right\|_{\ell^{\i}(\Z)}\left\|\left\{\|h\|_{1,v,I_{k+1}}
\right\}\right\|_{\ell^1(\Z)}\nonumber\\
& \hspace{-4cm} = \left\|\left\{\vp(x_k)C(x_k,x_{k+1})\right\}
\right\|_{\ell^{\i}(\Z)}\|h\|_{1,v,\I}. \label{IV}
\end{align}
Combining \eqref{III} and \eqref{IV}, in view of \eqref{eq.4.12},
we get that
\begin{align*}
\lhs \eqref{eq.4.2} & \nonumber\\
&\hspace{-1cm}\ls \left(\left\|\left\{ B(x_{k-1},x_k)
\frac{\vp(x_k)^{\frac{1}{q}}}{U(x_k)^{\frac{1}{p}}}\right\}\right\|_{\ell^{\i}(\Z)}
+ \left\|\left\{\vp(x_k)^{\frac{1}{q}}C(x_k,x_{k+1})\right\}
\right\|_{\ell^{\i}(\Z)}\right)\rhs \eqref{eq.4.2}.
\end{align*}
Consequently, inequality \eqref{eq.4.2} holds provided that
$D<\i$, and $c\ls D$.

Next we prove that condition \eqref{D} is also necessary for the
validity of inequality \eqref{eq.4.2}. Assume that inequality
\eqref{eq.4.2} holds with $c < \i$. By \eqref{hknorm}, \eqref{hk
saturates} and \eqref{g}, we obtain that
\begin{align}
\lhs \eqref{eq.4.2} & \gs
\left\|\left\{\left\|\int_s^{x_k}\sum_{m\in \Z} a_m
g_m\right\|_{p,u,I_k}\frac{\vp(x_k)}{U(x_k)^{\frac{1}{p}}}\right\}\right\|_{\ell^{\i}(\Z)}
\nonumber\\
& \gs \left\|\left\{a_k
B(x_{k-1},x_k)\frac{\vp(x_k)}{U(x_k)^{\frac{1}{p}}}\right\}\right\|_{\ell^{\i}(\Z)}.
\label{nec100}
\end{align}
Moreover,
\begin{align}\label{nec200}
\rhs \eqref{eq.4.2}  = c \left\|\sum_{m\in \Z} a_m
g_m\right\|_{1,v,\I} = c
\left\|\left\{a_k\right\}\right\|_{\ell^1(\Z)}.
\end{align}
Therefore, by \eqref{eq.4.2}, \eqref{nec100} and \eqref{nec200},
\begin{equation}
\left\|\left\{a_k
B(x_{k-1},x_k)\frac{\vp(x_k)}{U(x_k)^{\frac{1}{p}}}\right\}\right\|_{\ell^{\i}(\Z)}
\ls c \left\|\left\{a_k\right\}\right\|_{\ell^1(\Z)},
\end{equation}
and Proposition \ref{prop.2.1} implies that
\begin{equation}\label{eq256-8209682=4586}
\left\| \left\{
\frac{\vp(x_k)}{U(x_k)^{\frac{1}{p}}}B(x_{k-1},x_k)\right\}
\right\|_{\ell^{\i}(\Z)}  \ls  c.
\end{equation}

On the other hand, accordingly to \eqref{eq13487561134551},
\eqref{eq1349601=41345} and \eqref{f}, we obtain that
\begin{align*}
\lhs \eqref{eq.4.2} & \gs
\left\|\left\{\vp(x_k)\int_{x_k}^{x_{k+1}}\sum_{m\in \Z} b_m f_m
\right\}\right\|_{\ell^{\i}(\Z)} \gs \left\|\left\{b_k
\vp(x_k)C(x_k,x_{k+1})\right\}\right\|_{\ell^{\i}(\Z)}.
\end{align*}
Since,
\begin{align*}
\rhs \eqref{eq.4.2} = c \left\|\sum_{m\in \Z} b_m
f_m\right\|_{1,v,\I} = c
\left\|\left\{b_k\right\}\right\|_{\ell^1(\Z)},
\end{align*}
in view of \eqref{eq.4.2} and previous two inequalities, we have
that
$$
\left\|\left\{b_k
\vp(x_k)C(x_k,x_{k+1})\right\}\right\|_{\ell^{\i}(\Z)} \ls c
\left\|\left\{b_k\right\}\right\|_{\ell^1(\Z)}.
$$
Proposition \ref{prop.2.1} implies that
\begin{equation}\label{eq4096582=9667}
\left\|\left\{
\vp(x_k)C(x_k,x_{k+1})\right\}\right\|_{\ell^{\i}(\Z)} \ls c.
\end{equation}
Finally, inequalities \eqref{eq256-8209682=4586} and
\eqref{eq4096582=9667} imply that $D \ls c$.
\end{proof}

\begin{rem}\label{remconv}
In view of \eqref{Ck} and Lemma \ref{lem.2.3}, it is evident now
that
$$
\left\|\left\{\vp(x_k)^{\frac{1}{q}}C(x_k,x_{k+1})\right\}
\right\|_{\ell^{\rho}(\Z)} \ap \left\| \left\{
\vp(x_k)^{\frac{1}{q}} \underline{v}(x_k,x_{k+1})\right\}
\right\|_{\ell^{\rho}(\Z)} \ap \left\| \left\{
\vp(x_k)^{\frac{1}{q}} \underline{v}(x_k,\i)\right\}
\right\|_{\ell^{\rho}(\Z)}.
$$
Monotonicity of $\underline{v}(t,\i)$ implies that
\begin{align*}
\left\| \left\{ \vp(x_k)^{\frac{1}{q}}
\underline{v}(x_k,x_{k+1})\right\} \right\|_{\ell^{\rho}(\Z)}  \ge
\left\| \left\{ \vp(x_k)^{\frac{1}{q}} \right\}
\right\|_{\ell^{\rho}(\Z)} \lim_{t\rightarrow
\i}\underline{v}(t,\i).
\end{align*}
Since $\left\{ \vp(x_k)^{\frac{1}{q}} \right\}$ is geometrically
increasing, we obtain that
$$
\left\| \left\{ \vp(x_k)^{\frac{1}{q}}
\underline{v}(x_k,x_{k+1})\right\} \right\|_{\ell^{\rho}(\Z)}  \ge
\vp(\i)^{\frac{1}{q}} \lim_{t\rightarrow \i}\underline{v}(t,\i).
$$
This inequality shows that $\lim_{t\rightarrow
\i}\underline{v}(t,\i)$ must be equal to $0$, because $\vp(\i)$  is
always equal to $\i$ by our assumptions on the function $\vp$.
Therefore, in the remaining part of the paper we consider weight
functions $v$ such that
$$
\lim_{t\rightarrow \i}\underline{v}(t,\i) = 0.
$$
\end{rem}

\section{\bf Anti-dicretization of Conditions}\label{Antidiscretization}

In this section we anti-discretize the conditions obtained in Lemmas
\ref{lem14357861} and \ref{lem1435786143515}. We distinguish several
cases.

\

\noindent{\bf The case $0< p<1$, $0 < q < \i$.}  We need the
following lemma.
\begin{lem}\label{lem3.3}
Let $0<q<\i$, $0<p<1$, $1 / \rho = (1 / q - 1)_+$, and let
$u,\,v,\,w$ be weights. Assume that $u$ be such that $U$ is
admissible and the measure $w(t)dt$ is non-degenerate with respect
to $U^{\frac{q}{p}}$. Let $\{x_k\}$ be any discretizing sequence for
$\vp$ defined by \eqref{eq.4.3}. Then
$$
A  \ap A_1,
$$
where
\begin{equation*}
A_1 : =  \left\| \left\{
\frac{\vp(x_k)^{\frac{1}{q}}}{U(x_k)^{\frac{1}{p}}} \left(
\int_{x_{k-1}}^{x_k} \left( \int_{x_{k-1}}^t u(s)\,ds\right)^{p^*}
u(t) \underline{v}(t,\i)^{p^*} \,dt \right)^{\frac{1}{p^*}}\right\}
\right\|_{\ell^{\rho}(\Z)}.
\end{equation*}
\end{lem}
\begin{proof}
By Lemma \ref{lem.2.800}, in this case it yields that
\begin{equation*}
B(x_{k-1},x_k) \ap \left(\int_{x_{k-1}}^{x_k} \left(\int_{x_{k-1}}^t
u(s)\,ds\right)^{p^*} u(t)\underline{v}(t,x_k)^{p^*}\,dt \right)
^{\frac{1}{{p^*}}}.
\end{equation*}
Therefore, in view of \eqref{Ck}, Lemma \ref{lem14357861}, we have
that
\begin{align*}
A  \ap & \left\| \left\{
\frac{\vp(x_k)^{\frac{1}{q}}}{U(x_k)^{\frac{1}{p}}} \left(
\int_{x_{k-1}}^{x_k} \left( \int_{x_{k-1}}^t u(s)\,ds\right)^{p^*}
u(t) \underline{v}(t,x_k)^{p^*} \,dt
\right)^{\frac{1}{{p^*}}}\right\}
\right\|_{\ell^{\rho}(\Z)} \notag\\
& + \left\|\left\{ \vp(x_k)^{\frac{1}{q}}
\underline{v}(x_k,x_{k+1})\right\}\right\|_{\ell^{\rho}(\Z)}.
\end{align*}
It is easy to see that
\begin{align*}
A_1 \ls & \left\| \left\{
\frac{\vp(x_k)^{\frac{1}{q}}}{U(x_k)^{\frac{1}{p}}} \left(
\int_{x_{k-1}}^{x_k} \left( \int_{x_{k-1}}^t u(s)\,ds\right)^{p^*}
u(t) \underline{v}(t,x_k)^{p^*}\,dt
\right)^{\frac{1}{{p^*}}}\right\}
\right\|_{\ell^{\rho}(\Z)} \notag \\
& + \left\| \left\{
\frac{\vp(x_k)^{\frac{1}{q}}}{U(x_k)^{\frac{1}{p}}}
\underline{v}(x_k,\i)\left( \int_{x_{k-1}}^{x_k} \left(
\int_{x_{k-1}}^t u(s)\,ds\right)^{p^*} u(t) \,dt
\right)^{\frac{1}{{p^*}}}\right\}
\right\|_{\ell^{\rho}(\Z)} \notag \\
= & \left\| \left\{
\frac{\vp(x_k)^{\frac{1}{q}}}{U(x_k)^{\frac{1}{p}}} \left(
\int_{x_{k-1}}^{x_k} \left( \int_{x_{k-1}}^t u(s)\,ds\right)^{p^*}
u(t) \underline{v}(t,x_k)^{p^*} \,dt
\right)^{\frac{1}{{p^*}}}\right\}
\right\|_{\ell^{\rho}(\Z)} \notag \\
& + \left\| \left\{
\frac{\vp(x_k)^{\frac{1}{q}}}{U(x_k)^{\frac{1}{p}}}
\underline{v}(x_k,\i)\left( \int_{x_{k-1}}^{x_k}  u(t) \,dt
\right)^{\frac{1}{p}}\right\}
\right\|_{\ell^{\rho}(\Z)} \notag \\
\ap & \left\| \left\{
\frac{\vp(x_k)^{\frac{1}{q}}}{U(x_k)^{\frac{1}{p}}} \left(
\int_{x_{k-1}}^{x_k} \left( \int_{x_{k-1}}^t u(s)\,ds\right)^{p^*}
u(t) \underline{v}(t,x_k)^{p^*} \,dt
\right)^{\frac{1}{{p^*}}}\right\}
\right\|_{\ell^{\rho}(\Z)} \notag \\
& + \left\| \left\{ \vp(x_k)^{\frac{1}{q}}
\underline{v}(x_k,\i)\right\} \right\|_{\ell^{\rho}(\Z)} \\
\ap & \left\| \left\{
\frac{\vp(x_k)^{\frac{1}{q}}}{U(x_k)^{\frac{1}{p}}} \left(
\int_{x_{k-1}}^{x_k} \left( \int_{x_{k-1}}^t u(s)\,ds\right)^{p^*}
u(t) \underline{v}(t,x_k)^{p^*} \,dt
\right)^{\frac{1}{{p^*}}}\right\}
\right\|_{\ell^{\rho}(\Z)} \notag \\
& + \left\| \left\{ \vp(x_k)^{\frac{1}{q}}
\underline{v}(x_k,x_{k+1})\right\} \right\|_{\ell^{\rho}(\Z)} \ap A.
\end{align*}

On the other hand,
\begin{align*}
A \ap & \left\| \left\{
\frac{\vp(x_k)^{\frac{1}{q}}}{U(x_k)^{\frac{1}{p}}} \left(
\int_{x_{k-1}}^{x_k} \left( \int_{x_{k-1}}^t u(s)\,ds\right)^{p^*}
u(t) \underline{v}(t,x_k)^{p^*} \,dt
\right)^{\frac{1}{{p^*}}}\right\}
\right\|_{\ell^{\rho}(\Z)} \notag \\
& + \left\| \left\{ \vp(x_k)^{\frac{1}{q}}
\underline{v}(x_k,x_{k+1})\right\} \right\|_{\ell^{\rho}(\Z)} \notag
\\
\ap &  \left\| \left\{
\frac{\vp(x_k)^{\frac{1}{q}}}{U(x_k)^{\frac{1}{p}}} \left(
\int_{x_{k-1}}^{x_k} \left( \int_{x_{k-1}}^t u(s)\,ds\right)^{p^*}
u(t) \underline{v}(t,x_k)^{p^*} \,dt \right)^{\frac{1}{p^*}}\right\}
\right\|_{\ell^{\rho}(\Z)} \notag \\
& + \left\| \left\{
\frac{\vp(x_k)^{\frac{1}{q}}}{U(x_k)^{\frac{1}{p}}}\underline{v}(x_k,x_{k+1})\left(
\int_{x_{k-1}}^{x_k} \left( \int_{x_{k-1}}^t u(s)\,ds\right)^{p^*}
u(t)
\,dt \right)^{\frac{1}{p^*}} \right\} \right\|_{\ell^{\rho}(\Z)} \notag \\
\ls & \left\| \left\{
\frac{\vp(x_k)^{\frac{1}{q}}}{U(x_k)^{\frac{1}{p}}} \left(
\int_{x_{k-1}}^{x_k} \left( \int_{x_{k-1}}^t u(s)\,ds\right)^{p^*}
u(t) \underline{v}(t,\i)^{p^*} \,dt \right)^{\frac{1}{p^*}}\right\}
\right\|_{\ell^{\rho}(\Z)} = A_1.
\end{align*}
\end{proof}

\begin{lem}\label{lem3.4}
Assume that the conditions of Lemma \ref{lem3.3} are fulfilled. Then
$$
A_1 \ap A_2,
$$
where
\begin{equation*}
A_2 : = \left\| \left\{
\frac{\vp(x_k)^{\frac{1}{q}}}{U(x_k)^{\frac{1}{p}}} \left(
\int_{x_{k-1}}^{x_k} U(t)^{p^*} u(t) \underline{v}(t,\i)^{p^*} \,dt
\right)^{\frac{1}{p^*}}\right\} \right\|_{\ell^{\rho}(\Z)}.
\end{equation*}
\end{lem}
\begin{proof}
Evidently, $A_1 \le A_2$. Using integrating by parts formula
\eqref{integration.by.parts}, we have that
\begin{align*}
A_2 \ap & \left\| \left\{
\frac{\vp(x_k)^{\frac{1}{q}}}{U(x_k)^{\frac{1}{p}}} \left(
\int_{[x_{k-1},x_k)} \underline{v}(t,\i)^{p^*}
d\left(U(t)^{\frac{p^*}{p}} \right) \right)^{\frac{1}{p^*}}\right\}
\right\|_{\ell^{\rho}(\Z)} \notag \\
\le & \left\| \left\{
\frac{\vp(x_k)^{\frac{1}{q}}}{U(x_k)^{\frac{1}{p}}} \left(
\int_{[x_{k-1},x_k)} U(t)^{\frac{p^*}{p}} d\left( -
\underline{v}(t-,\i)^{p^*} \right) \right)^{\frac{1}{p^*}}\right\}
\right\|_{\ell^{\rho}(\Z)} \notag \\
& + \left\| \left\{ \vp(x_k)^{\frac{1}{q}}
\underline{v}(x_k-,\i)\right\}
\right\|_{\ell^{\rho}(\Z)} \notag \\
\le & \left\| \left\{
\frac{\vp(x_k)^{\frac{1}{q}}}{U(x_k)^{\frac{1}{p}}} \left(
\int_{[x_{k-1},x_k)} \left( \int_{x_{k-1}}^t
u(s)\,ds\right)^{\frac{p^*}{p}} d\left( - \underline{v}(t-,\i)^{p^*}
\right)\right)^{\frac{1}{p^*}}\right\}
\right\|_{\ell^{\rho}(\Z)} \notag \\
& + \left\| \left\{
\frac{\vp(x_k)^{\frac{1}{q}}}{U(x_k)^{\frac{1}{p}}}
U(x_{k-1})^{\frac{1}{p}}\left( \int_{[x_{k-1},x_k)}  d\left( -
\underline{v}(t-,\i)^{p^*} \right) \right)^{\frac{1}{p^*}}\right\}
\right\|_{\ell^{\rho}(\Z)} \notag \\
& + \left\| \left\{ \vp(x_k)^{\frac{1}{q}}
\underline{v}(x_k-,\i)\right\}
\right\|_{\ell^{\rho}(\Z)} \notag \\
\le & \left\| \left\{
\frac{\vp(x_k)^{\frac{1}{q}}}{U(x_k)^{\frac{1}{p}}} \left(
\int_{[x_{k-1},x_k)} \left( \int_{x_{k-1}}^t
u(s)\,ds\right)^{\frac{p^*}{p}} d\left( - \underline{v}(t-,\i)^{p^*}
\right) \right)^{\frac{1}{p^*}}\right\}
\right\|_{\ell^{\rho}(\Z)} \notag \\
& + \left\| \left\{
\frac{\vp(x_k)^{\frac{1}{q}}}{U(x_k)^{\frac{1}{p}}}
U(x_{k-1})^{\frac{1}{p}} \underline{v}(x_{k-1}-,\i)  \right\}
\right\|_{\ell^{\rho}(\Z)} \\
& + \left\| \left\{ \vp(x_k)^{\frac{1}{q}}
\underline{v}(x_k-,\i)\right\}
\right\|_{\ell^{\rho}(\Z)} \notag \\
\ls & \left\| \left\{
\frac{\vp(x_k)^{\frac{1}{q}}}{U(x_k)^{\frac{1}{p}}} \left(
\int_{[x_{k-1},x_k)} \left( \int_{x_{k-1}}^t
u(s)\,ds\right)^{\frac{p^*}{p}} d\left( - \underline{v}(t-,\i)^{p^*}
\right) \right)^{\frac{1}{p^*}}\right\}
\right\|_{\ell^{\rho}(\Z)} \notag \\
& + \left\| \left\{ \vp(x_{k-1})^{\frac{1}{q}}
\underline{v}(x_{k-1}-,\i) \right\} \right\|_{\ell^{\rho}(\Z)} +
\left\| \left\{ \vp(x_k)^{\frac{1}{q}}
\underline{v}(x_k-,\i)\right\}
\right\|_{\ell^{\rho}(\Z)} \notag \\
\ap & \left\| \left\{
\frac{\vp(x_k)^{\frac{1}{q}}}{U(x_k)^{\frac{1}{p}}} \left(
\int_{[x_{k-1},x_k)} \left( \int_{x_{k-1}}^t
u(s)\,ds\right)^{\frac{p^*}{p}} d\left( - \underline{v}(t-,\i)^{p^*}
\right) \right)^{\frac{1}{p^*}}\right\}
\right\|_{\ell^{\rho}(\Z)} \notag \\
& + \left\| \left\{ \vp(x_k)^{\frac{1}{q}}
\underline{v}(x_k-,\i)\right\} \right\|_{\ell^{\rho}(\Z)}.
\end{align*}
Again integrating by parts we have that
\begin{align*}
A_2 \ls & \left\| \left\{
\frac{\vp(x_k)^{\frac{1}{q}}}{U(x_k)^{\frac{1}{p}}} \left(
\int_{x_{k-1}}^{x_k} \underline{v}(t-,\i)^{p^*} d\left(
\int_{x_{k-1}}^t u(s)\,ds\right)^{\frac{p^*}{p}}
\right)^{\frac{1}{p^*}}\right\}
\right\|_{\ell^{\rho}(\Z)} \notag \\
& + \left\| \left\{ \vp(x_k)^{\frac{1}{q}}
\underline{v}(x_k-,\i)\right\}
\right\|_{\ell^{\rho}(\Z)}\notag\\
= & \left\| \left\{
\frac{\vp(x_k)^{\frac{1}{q}}}{U(x_k)^{\frac{1}{p}}} \left(
\int_{x_{k-1}}^{x_k} \left( \int_{x_{k-1}}^t u(s)\,ds\right)^{p^*}
u(t) \underline{v}(t-,\i)^{p^*} \,dt \right)^{\frac{1}{p^*}}\right\}
\right\|_{\ell^{\rho}(\Z)} \notag \\
& + \left\| \left\{ \vp(x_k)^{\frac{1}{q}}
\underline{v}(x_k-,\i)\right\}
\right\|_{\ell^{\rho}(\Z)}\notag\\
= & A_1 + \left\| \left\{ \vp(x_k)^{\frac{1}{q}}
\underline{v}(x_k-,\i)\right\} \right\|_{\ell^{\rho}(\Z)}.
\end{align*}
Since
\begin{align}
\left\| \left\{ \vp(x_k)^{\frac{1}{q}}
\underline{v}(x_k-,\i)\right\}
\right\|_{\ell^{\rho}(\Z)} & \notag\\
& \hspace{-4cm} =  \left\| \left\{ \vp(x_{k-1})^{\frac{1}{q}}
\underline{v}(x_{k-1}-,\i)\right\} \right\|_{\ell^{\rho}(\Z)} \notag \\
& \hspace{-4cm} \ap \left\| \left\{
\frac{\vp(x_{k-1})^{\frac{1}{q}}}{U(x_{k-1})^{\frac{1}{p}}}
\underline{v}(x_{k-1}-,\i) \left( \int_{x_{k-2}}^{x_{k-1}} \left(
\int_{x_{k-2}}^t u(s)\,ds\right)^{p^*}
u(t)\,dt\right)^{\frac{1}{p^*}}\right\}\right\|_{\ell^{\rho}(\Z)} \notag \\
& \hspace{-4cm} \le \left\| \left\{
\frac{\vp(x_{k-1})^{\frac{1}{q}}}{U(x_{k-1})^{\frac{1}{p}}} \left(
\int_{x_{k-2}}^{x_{k-1}} \left( \int_{x_{k-2}}^t
u(s)\,ds\right)^{p^*}
u(t)\underline{v}(t-,\i)\,dt\right)^{\frac{1}{p^*}}\right\}\right\|_{\ell^{\rho}(\Z)} \notag \\
& \hspace{-4cm} = \left\| \left\{
\frac{\vp(x_k)^{\frac{1}{q}}}{U(x_k)^{\frac{1}{p}}} \left(
\int_{x_{k-1}}^{x_k} \left( \int_{x_{k-1}}^t u(s)\,ds\right)^{p^*}
u(t) \underline{v}(t-,\i)^{p^*} \,dt \right)^{\frac{1}{p^*}}\right\}
\right\|_{\ell^{\rho}(\Z)} = A_1, \label{eq2-568259682=}
\end{align}
we obtain that
$$
A_2 \ls A_1.
$$
\end{proof}

\begin{lem}\label{lem3.5}
Assume that the conditions of Lemma \ref{lem3.3} are fulfilled. Then
$$
A_2 \ap A_3,
$$
where
\begin{align*}
A_3 : = & \left\| \left\{
\frac{\vp(x_k)^{\frac{1}{q}}}{U(x_k)^{\frac{1}{p}}} \left(
\int_{[x_{k-1},x_k)} U(t)^{\frac{p^*}{p}} d\left( -
\underline{v}(t-,\i)^{p^*} \right) \,dt
\right)^{\frac{1}{p^*}}\right\} \right\|_{\ell^{\rho}(\Z)} \notag \\
& + \left\| \left\{
\vp(x_k)^{\frac{1}{q}}\underline{v}(x_k-,\i\right\}
\right\|_{\ell^{\rho}(\Z)}.
\end{align*}
\end{lem}
\begin{proof}
Integrating by parts, in view of inequality \eqref{eq2-568259682=}
and Lemma \ref{lem3.4}, we have that
\begin{align*}
A_3  \le & \left\| \left\{
\frac{\vp(x_k)^{\frac{1}{q}}}{U(x_k)^{\frac{1}{p}}} \left(
\int_{x_{k-1}}^{x_k} \underline{v}(t,\i)^{p^*} d\left(
U(t)^{\frac{p^*}{p}} \right) \,dt
\right)^{\frac{1}{p^*}}\right\} \right\|_{\ell^{\rho}(\Z)} \notag \\
& + \left\| \left\{
\frac{\vp(x_k)^{\frac{1}{q}}}{U(x_k)^{\frac{1}{p}}}
U(x_{k-1})^{\frac{1}{p}}\underline{v}(x_{k-1}-,\i)  \right\}
\right\|_{\ell^{\rho}(\Z)} + \left\| \left\{
\vp(x_k)^{\frac{1}{q}}\underline{v}(x_k-,\i\right\}
\right\|_{\ell^{\rho}(\Z)} \notag \\
\ls & A_2 + \left\| \left\{
\vp(x_{k-1})^{\frac{1}{q}}\underline{v}(x_{k-1}-,\i\right\}
\right\|_{\ell^{\rho}(\Z)} + \left\| \left\{
\vp(x_k)^{\frac{1}{q}}\underline{v}(x_k-,\i\right\}
\right\|_{\ell^{\rho}(\Z)}\notag \\
\ap & A_2 + \left\| \left\{
\vp(x_k)^{\frac{1}{q}}\underline{v}(x_k-,\i\right\}
\right\|_{\ell^{\rho}(\Z)} \ls A_2 + A_1 \ap A_2.
\end{align*}
On the other hand, again integrating by parts, we get that
\begin{align*}
A_2 = & \left\| \left\{
\frac{\vp(x_k)^{\frac{1}{q}}}{U(x_k)^{\frac{1}{p}}}
\left(\int_{x_{k-1}}^{x_k}\underline{v}(t,\i)^{p^*} \,d\left(
U(t)^{\frac{p^*}{p}}\right)
\right)^{\frac{1}{p^*}}\right\} \right\|_{\ell^{\rho}(\Z)} \notag \\
\ls & \left\| \left\{
\frac{\vp(x_k)^{\frac{1}{q}}}{U(x_k)^{\frac{1}{p}}} \left(
\int_{[x_{k-1},x_k)} U(t)^{\frac{p^*}{p}} d\left( -
\underline{v}(t-,\i)^{p^*} \right)
\right)^{\frac{1}{p^*}}\right\} \right\|_{\ell^{\rho}(\Z)} \notag \\
& + \left\| \left\{
\vp(x_k)^{\frac{1}{q}}\underline{v}(x_k-,\i\right\}
\right\|_{\ell^{\rho}(\Z)} = A_3.
\end{align*}
\end{proof}

\begin{lem}\label{lem3.6}
Assume that the conditions of Lemma \ref{lem3.3} are fulfilled. Then
$$
A_3 \ap A_4,
$$
where
\begin{align*}
A_4 : = & \left\| \left\{
\frac{\vp(x_k)^{\frac{1}{q}}}{U(x_k)^{\frac{1}{p}}} \left(
\int_{[x_{k-1},x_k)} U(t)^{\frac{p^*}{p}} d\left( -
\underline{v}(t-,\i)^{p^*} \right)
\right)^{\frac{1}{p^*}}\right\} \right\|_{\ell^{\rho}(\Z)} \notag \\
& + \left\| \left\{ \vp(x_k)^{\frac{1}{q}}
\left(\int_{[x_k,x_{k+1})} d \left( -
\underline{v}(t-,\i)^{p^*}\right) \right)^{\frac{1}{p^*}}\right\}
\right\|_{\ell^{\rho}(\Z)}.
\end{align*}
\end{lem}
\begin{proof}
By Lemma \ref{lem.2.3}, in view of Remark \ref{remconv}, we have
that
\begin{align*}
\left\| \left\{ \vp(x_k)^{\frac{1}{q}}\underline{v}(x_k-,\i\right\}
\right\|_{\ell^{\rho}(\Z)} & \notag \\
& \hspace{-3cm}\ap \left\| \left\{ \vp(x_k)^{\frac{1}{q}}\left(
\sum_{i=k}^{\i} \left[\underline{v}(x_i-,\i)^{p^*} -
\underline{v}(x_{i+1}-,\i)^{p^*}
\right]\right)^{\frac{1}{p^*}} \right\} \right\|_{\ell^{\rho}(\Z)} \notag\\
& \hspace{-2.5cm} + \left\| \left\{
\vp(x_k)^{\frac{1}{q}}\lim_{t\rightarrow\i}\underline{v}(t,\i)
\right\}
\right\|_{\ell^{\rho}(\Z)} \notag \\
& \hspace{-3cm}\ap \left\| \left\{ \vp(x_k)^{\frac{1}{q}}\left(
\underline{v}(x_k-,\i)^{p^*} - \underline{v}(x_{k+1}-,\i)^{p^*}
\right)^{\frac{1}{p^*}} \right\} \right\|_{\ell^{\rho}(\Z)} \notag\\
& \hspace{-3cm}\ap \left\| \left\{ \vp(x_k)^{\frac{1}{q}}\left(
\int_{[x_k,x_{k+1})} d \left( - \underline{v}(t-,\i)^{p^*}\right)
\right)^{\frac{1}{p^*}} \right\} \right\|_{\ell^{\rho}(\Z)}.
\end{align*}
\end{proof}

\begin{lem}\label{lem3.60}
Assume that the conditions of Lemma \ref{lem3.3} are fulfilled. Then
$$
A_4 \ap A_5,
$$
where
\begin{align*}
A_5 : =  \left\| \left\{ \vp(x_k)^{\frac{1}{q}} \left(\int_{[0,\i)}
{\mathcal U}(t,x_k)^{\frac{p^*}{p}}d \left( -
\underline{v}(t-,\i)^{p^*}\right) \right)^{\frac{1}{p^*}}\right\}
\right\|_{\ell^{\rho}(\Z)}.
\end{align*}
\end{lem}
\begin{proof}
By Lemma \ref{lem.2.3}, we have that
\begin{align*}
A_4  \ap & \left\| \left\{
\frac{\vp(x_k)^{\frac{1}{q}}}{U(x_k)^{\frac{1}{p}}} \left(
\int_{[0,x_k)} U(t)^{\frac{p^*}{p}} d\left( -
\underline{v}(t-,\i)^{p^*} \right)
\right)^{\frac{1}{p^*}}\right\} \right\|_{\ell^{\rho}(\Z)} \notag \\
& + \left\| \left\{ \vp(x_k)^{\frac{1}{q}} \left(\int_{[x_k,\i)} d
\left( - \underline{v}(t-,\i)^{p^*}\right)
\right)^{\frac{1}{p^*}}\right\} \right\|_{\ell^{\rho}(\Z)}.
\end{align*}
Hence
\begin{align*}
A_4  \ap & \left\| \left\{ \vp(x_k)^{\frac{1}{q}} \left(
\int_{[0,x_k)} {\mathcal U}(t,x_k)^{\frac{p^*}{p}} d\left( -
\underline{v}(t-,\i)^{p^*} \right)
\right)^{\frac{1}{p^*}}\right\} \right\|_{\ell^{\rho}(\Z)} \notag \\
& + \left\| \left\{ \vp(x_k)^{\frac{1}{q}} \left(\int_{[x_k,\i)}
{\mathcal U}(t,x_k)^{\frac{p^*}{p}}\,d \left( -
\underline{v}(t-,\i)^{p^*}\right)
\right)^{\frac{1}{p^*}}\right\} \right\|_{\ell^{\rho}(\Z)}\\
\ap & \left\| \left\{ \vp(x_k)^{\frac{1}{q}} \left(\int_{[0,\i)}
{\mathcal U}(t,x_k)^{\frac{p^*}{p}}\,d \left( -
\underline{v}(t-,\i)^{p^*}\right) \right)^{\frac{1}{p^*}}\right\}
\right\|_{\ell^{\rho}(\Z)} = A_5.
\end{align*}
\end{proof}

We are now in position to state and prove our first main theorem.
\begin{thm}\label{main1001}
Let $0<p<1$, $0 < q < \i$, and let $u,\,v,\,w$ be weights. Assume
that $u$ is such that $U$ is admissible and the measure $w(t)dt$ is
non-degenerate with respect to $U^{\frac{q}{p}}$.

{\rm (i)} Let $1\leq q<\i$. Then inequality \eqref{eq.4.1} holds
for every $h \in \mp^+(0,\i)$ if and only if
\begin{align*}
I_1: = \sup_{x\in (0,\i)} \left(\int_0^{\i}{\mathcal
U}(x,s)^{\frac{q}{p}}w(s)ds\right)^{\frac{1}{q}} \,\left(
\int_{[0,\i)} {\mathcal U}(t,x)^{\frac{p^*}{p}}d\left( -
\underline{v}(t-,\i)^{p^*} \right)\right)^{\frac{1}{p^*}}<\i.
\end{align*}
Moreover, the best constant $c$ in \eqref{eq.4.1} satisfies $c\ap
I_1$.

{\rm (ii)} Let $0<q<1$. Then inequality \eqref{eq.4.1} holds for
every $h \in \mp^+(0,\i)$ if and only if
\begin{align*}
I_2:= \left(\int_0^{\i}\left(\int_0^{\i}{\mathcal
U}(x,s)^{\frac{q}{p}}w(s)ds\right)^{q^*}
\left(\int_{[0,\i)}{\mathcal
U}(t,x)^{\frac{p^*}{p}}d\left(-\underline{v}(t-,\i)^{p^*}\right)
\right)^{\frac{q^*}{p^*}}w(x)\,dx\right)^{\frac{1}{q^*}}  <\i.
\end{align*}
Moreover, the best constant $c$ in \eqref{eq.4.1} satisfies $c\ap
I_2$.
\end{thm}

\begin{proof}
{\rm (i)} The proof of the statement follows by using Lemmas
\ref{lem14357861}, \ref{lem3.3}-\ref{lem3.60} and \ref{lem.2.2}.

{\rm (ii)} The proof of the statement follows by combining Lemmas
\ref{lem14357861}, \ref{lem3.3}-\ref{lem3.60} and Theorem
\ref{thm.2.1}.
\end{proof}

\

\noindent{\bf The case $1 \le p < \i$, $0 < q < \i$.} The following
lemma is true.
\begin{lem}\label{lem3.8}
Let $1\le p<\i$, $0<q<\i$ and let $u,\,v,\,w$ be weights. Assume
that $u$ is such that $U$ is admissible and the measure $w(t)dt$ is
non-degenerate with respect to $U^{\frac{q}{p}}$. Let $\{x_k\}$ be
any discretizing sequence for $\vp$ defined by \eqref{eq.4.3}. Then
$$
A \ap B_1,
$$
where
\begin{align*}
B_1  : =  \left\| \left\{
\frac{\vp(x_k)^{\frac{1}{q}}}{U(x_k)^{\frac{1}{p}}} \left(
\sup_{x_{k-1}<t<x_k} \left( \int_{x_{k-1}}^t
u(s)\,ds\right)^{\frac{1}{p}} \underline{v}(t,\i) \right)\right\}
\right\|_{\ell^{\rho}(\Z)}.
\end{align*}
\end{lem}
\begin{proof}
By Lemma \ref{lem.2.800}, in this case we find that
\begin{equation*}
B(x_{k-1},x_k) \ap \sup_{x_{k-1}<t<x_k} \left(\int_{x_{k-1}}^t
u(s)\,ds\right)^{\frac{1}{p}} \underline{v}(t,x_k).
\end{equation*}
By using \eqref{Ck}, in view of Lemma \ref{lem14357861}, we have
that
\begin{align*}
A  \ap & \left\| \left\{
\frac{\vp(x_k)^{\frac{1}{q}}}{U(x_k)^{\frac{1}{p}}} \left(
\sup_{x_{k-1}<t<x_k} \left( \int_{x_{k-1}}^t
u(s)\,ds\right)^{\frac{1}{p}} \underline{v}(t,x_k) \right)\right\}
\right\|_{\ell^{\rho}(\Z)} \notag\\
& + \left\|\left\{ \vp(x_k)^{\frac{1}{q}}
\underline{v}(x_k,x_{k+1})\right\}\right\|_{\ell^{\rho}(\Z)}.
\end{align*}
Obviously,
\begin{align*}
B_1 \ls & \left\| \left\{
\frac{\vp(x_k)^{\frac{1}{q}}}{U(x_k)^{\frac{1}{p}}} \left(
\sup_{x_{k-1}<t<x_k} \left( \int_{x_{k-1}}^t
u(s)\,ds\right)^{\frac{1}{p}} \underline{v}(t,x_k) \right)\right\}
\right\|_{\ell^{\rho}(\Z)} \notag \\
& + \left\| \left\{
\frac{\vp(x_k)^{\frac{1}{q}}}{U(x_k)^{\frac{1}{p}}}
\underline{v}(x_k,\i)\left( \sup_{x_{k-1}<t<x_k} \left(
\int_{x_{k-1}}^t u(s)\,ds\right)^{\frac{1}{p}}  \right)\right\}
\right\|_{\ell^{\rho}(\Z)}
\notag \\
= & \left\| \left\{
\frac{\vp(x_k)^{\frac{1}{q}}}{U(x_k)^{\frac{1}{p}}} \left(
\sup_{x_{k-1}<t<x_k} \left( \int_{x_{k-1}}^t
u(s)\,ds\right)^{\frac{1}{p}} \underline{v}(t,x_k) \right)\right\}
\right\|_{\ell^{\rho}(\Z)} \notag \\
& + \left\| \left\{
\frac{\vp(x_k)^{\frac{1}{q}}}{U(x_k)^{\frac{1}{p}}}
\underline{v}(x_k,\i)\left( \int_{x_{k-1}}^{x_k} u(s)\,ds
\right)^{\frac{1}{p}}\right\} \right\|_{\ell^{\rho}(\Z)}
\notag \\
\ap & \left\| \left\{
\frac{\vp(x_k)^{\frac{1}{q}}}{U(x_k)^{\frac{1}{p}}} \left(
\sup_{x_{k-1}<t<x_k} \left( \int_{x_{k-1}}^t
u(s)\,ds\right)^{\frac{1}{p}} \underline{v}(t,x_k) \right)\right\}
\right\|_{\ell^{\rho}(\Z)} \notag \\
& + \left\| \left\{ \vp(x_k)^{\frac{1}{q}}
\underline{v}(x_k,\i)\right\} \right\|_{\ell^{\rho}(\Z)}
\notag \\
\ap & \left\| \left\{
\frac{\vp(x_k)^{\frac{1}{q}}}{U(x_k)^{\frac{1}{p}}} \left(
\sup_{x_{k-1}<t<x_k} \left( \int_{x_{k-1}}^t
u(s)\,ds\right)^{\frac{1}{p}} \underline{v}(t,x_k) \right)\right\}
\right\|_{\ell^{\rho}(\Z)} \notag \\
& + \left\| \left\{ \vp(x_k)^{\frac{1}{q}}
\underline{v}(x_k,x_{k+1})\right\} \right\|_{\ell^{\rho}(\Z)} = A.
\end{align*}
On the other hand,
\begin{align*}
A \ap & \left\| \left\{
\frac{\vp(x_k)^{\frac{1}{q}}}{U(x_k)^{\frac{1}{p}}} \left(
\sup_{x_{k-1}<t<x_k} \left( \int_{x_{k-1}}^t
u(s)\,ds\right)^{\frac{1}{p}} \underline{v}(t,x_k) \right)\right\}
\right\|_{\ell^{\rho}(\Z)} \notag \\
& + \left\| \left\{
\frac{\vp(x_k)^{\frac{1}{q}}}{U(x_k)^{\frac{1}{p}}}
\underline{v}(x_k,x_{k+1})\left( \sup_{x_{k-1}<t<x_k} \left(
\int_{x_{k-1}}^t u(s)\,ds\right)^{\frac{1}{p}}  \right)\right\}
\right\|_{\ell^{\rho}(\Z)}
\notag \\
\ls & \left\| \left\{
\frac{\vp(x_k)^{\frac{1}{q}}}{U(x_k)^{\frac{1}{p}}} \left(
\sup_{x_{k-1}<t<x_k} \left( \int_{x_{k-1}}^t
u(s)\,ds\right)^{\frac{1}{p}} \underline{v}(t,\i) \right)\right\}
\right\|_{\ell^{\rho}(\Z)} : = B_1.
\end{align*}
\end{proof}

\begin{lem}\label{lem3.9}
Assume that the conditions of Lemma \ref{lem3.8} are fulfilled. Then
$$
B_1 \ap B_2,
$$
where
\begin{align*}
B_2  : =  \left\| \left\{
\frac{\vp(x_k)^{\frac{1}{q}}}{U(x_k)^{\frac{1}{p}}} \left(
\sup_{x_{k-1}<t<x_k} U(t)^{\frac{1}{p}} \underline{v}(t,\i)
\right)\right\} \right\|_{\ell^{\rho}(\Z)}.
\end{align*}
\end{lem}
\begin{proof}
Obviously,
$$
B_1 \le B_2.
$$
Since
\begin{align}
\left\| \left\{ \vp(x_k)^{\frac{1}{q}} \underline{v}(x_k,\i)
\right\}
\right\|_{\ell^{\rho}(\Z)} & \notag \\
& \hspace{-3cm} \ap \left\| \left\{
\frac{\vp(x_k)^{\frac{1}{q}}}{U(x_k)^{\frac{1}{p}}}\underline{v}(x_k,\i)
\left( \int_{x_{k-1}}^{x_k} u(s)\,ds  \right)^{\frac{1}{p}}\right\}
\right\|_{\ell^{\rho}(\Z)} \notag \\
& \hspace{-3cm}  = \left\| \left\{
\frac{\vp(x_k)^{\frac{1}{q}}}{U(x_k)^{\frac{1}{p}}}\underline{v}(x_k,\i)
\sup_{x_{k-1} < t < x_k}\left( \int_{x_{k-1}}^t u(s)\,ds
\right)^{\frac{1}{p}}\right\}
\right\|_{\ell^{\rho}(\Z)} \notag \\
& \hspace{-3cm}  \le \left\| \left\{
\frac{\vp(x_k)^{\frac{1}{q}}}{U(x_k)^{\frac{1}{p}}} \sup_{x_{k-1} <
t < x_k}\left( \int_{x_{k-1}}^t u(s)\,ds
\right)^{\frac{1}{p}}\underline{v}(t,\i)\right\}
\right\|_{\ell^{\rho}(\Z)} = B_1, \label{eq245-7682=9045687}
\end{align}
we obtain that
\begin{align*}
B_2 \le & B_1 + \left\| \left\{
\frac{\vp(x_k)^{\frac{1}{q}}}{U(x_k)^{\frac{1}{p}}}
U(x_{k-1})^{\frac{1}{p}}\sup_{x_{k-1} < t <
x_k}\underline{v}(t,\i)\right\}
\right\|_{\ell^{\rho}(\Z)} \notag \\
= & B_1 + \left\| \left\{
\frac{\vp(x_k)^{\frac{1}{q}}}{U(x_k)^{\frac{1}{p}}}
U(x_{k-1})^{\frac{1}{p}}\underline{v}(x_{k-1},\i)\right\}
\right\|_{\ell^{\rho}(\Z)} \notag \\
\ls & B_1 + \left\| \left\{
\vp(x_{k-1})^{\frac{1}{q}}\underline{v}(x_{k-1},\i)\right\}
\right\|_{\ell^{\rho}(\Z)} \notag \\
= & B_1  + \left\| \left\{
\vp(x_k)^{\frac{1}{q}}\underline{v}(x_k,\i)\right\}
\right\|_{\ell^{\rho}(\Z)} \ls B_1.
\end{align*}
\end{proof}

\begin{lem}\label{lem3.90}
Assume that the conditions of Lemma \ref{lem3.8} are fulfilled. Then
$$
B_2 \ap B_3,
$$
where
\begin{align*}
B_3  : =  \left\| \left\{
\frac{\vp(x_k)^{\frac{1}{q}}}{U(x_k)^{\frac{1}{p}}} \left(
\sup_{x_{k-1}<t<x_k} U(t)^{\frac{1}{p}} \underline{v}(t,\i)
\right)\right\} \right\|_{\ell^{\rho}(\Z)}  + \left\| \left\{
\vp(x_k)^{\frac{1}{q}} \underline{v}(x_k,\i)\right\}
\right\|_{\ell^{\rho}(\Z)}.
\end{align*}
\end{lem}
\begin{proof}
Obviously,
$$
B_2 \le B_3.
$$
On the other hand, by \eqref{eq245-7682=9045687}, we get that
\begin{align*}
\left\| \left\{ \vp(x_k)^{\frac{1}{q}}
\underline{v}(x_k,x_{k+1})\right\} \right\|_{\ell^{\rho}(\Z)} \ap
\left\| \left\{ \vp(x_k)^{\frac{1}{q}} \underline{v}(x_k,\i)
\right\} \right\|_{\ell^{\rho}(\Z)} \ls B_1 \ls B_2.
\end{align*}
Thus
\begin{align*}
B_3  : =  B_2  + \left\| \left\{ \vp(x_k)^{\frac{1}{q}}
\underline{v}(x_k,\i)\right\} \right\|_{\ell^{\rho}(\Z)} \ls B_2.
\end{align*}
\end{proof}

\begin{lem}\label{lem3.900}
Assume that the conditions of Lemma \ref{lem3.8} are fulfilled. Then
$$
B_3 \ap B_4,
$$
where
\begin{align*}
B_4 : = \left\| \left\{ \vp(x_k)^{\frac{1}{q}} \left( \sup_{t \in
(0,\i)} {\mathcal U}(t,x_k)^{\frac{1}{p}}
\underline{v}(t,\i)\right)\right\} \right\|_{\ell^{\rho}(\Z)}.
\end{align*}
\end{lem}
\begin{proof}
By Lemma \ref{lem.2.3}, we get that
\begin{align*}
B_3  \ap  & \left\| \left\{
\frac{\vp(x_k)^{\frac{1}{q}}}{U(x_k)^{\frac{1}{p}}} \left(
\sup_{0<t<x_k} U(t)^{\frac{1}{p}} \underline{v}(t,\i)
\right)\right\}
\right\|_{\ell^{\rho}(\Z)} \notag\\
& + \left\| \left\{ \vp(x_k)^{\frac{1}{q}} \left(\sup_{x_k < t <
\i}\underline{v}(t,\i)\right)\right\} \right\|_{\ell^{\rho}(\Z)} \notag\\
\ap  & \left\| \left\{ \vp(x_k)^{\frac{1}{q}} \left( \sup_{0<t<x_k}
{\mathcal U}(t,x_k)^{\frac{1}{p}} \underline{v}(t,\i)
\right)\right\}
\right\|_{\ell^{\rho}(\Z)} \notag\\
& + \left\| \left\{ \vp(x_k)^{\frac{1}{q}} \left( \sup_{x_k < t <
\i}{\mathcal U}(t,x_k)^{\frac{1}{p}}
\underline{v}(t,\i)\right)\right\}
\right\|_{\ell^{\rho}(\Z)} \notag\\
\ap & \left\| \left\{ \vp(x_k)^{\frac{1}{q}} \left(\sup_{t \in
(0,\i)} {\mathcal
U}(t,x_k)^{\frac{1}{p}}\underline{v}(t,\i)\right)\right\}
\right\|_{\ell^{\rho}(\Z)} = B_4.
\end{align*}
\end{proof}

Our next main result reads:
\begin{thm}\label{main1002}
Let $1 \le p < \i$, $0<q<\i$, and let $u,\,v,\,w$ be weights. Assume
that $u$ is such that $U$ is admissible and the measure $w(t)dt$ is
non-degenerate with respect to $U^{\frac{q}{p}}$.

{\rm (i)} Let $1\leq q <\i$. Then inequality \eqref{eq.4.1} holds
for every $h \in \mp^+(0,\i)$ if and only if
$$
I_3:=\sup_{x>0}\left(\int_0^{\i}{\mathcal
U}(x,t)^{\frac{q}{p}}w(t)dt\right)^{\frac{1}{q}}
{U(x)}^{-\frac{1}{p}}\sup_{t \in
(0,x)}{U(t)}^{\frac{1}{p}}{\underline{v}(t,\i)}<\i.
$$
Moreover, the best constant $c$ in \eqref{eq.4.1} satisfies that
$c\ap I_3$.

{\rm (ii)} Let $0<q<1$. Then inequality \eqref{eq.4.1} holds for
every $h \in \mp^+(0,\i)$ if and only if
\begin{gather*}
I_4:=\left(\int_0^{\i} \left(\int_0^{\i}{\mathcal
U}(x,t)^{\frac{q}{p}}w(t)dt\right)^{q^*} U(x)^{-\frac{q^*}{p}}
\left(\sup_{t\in(0,x)}U(t)^{\frac{q^*}{p}}{\underline{v}(t,\i)^{q^*}}\right){w(x)}\,dx\right)^{\frac{1}{q^*}}<\i.
\end{gather*}
Moreover, the best constant $c$ in \eqref{eq.4.1} satisfies that
$c\ap I_4$.
\end{thm}

\begin{proof}
{\rm (i)} The proof of the statement follows by combining Lemmas
\ref{lem3.8} - \ref{lem3.900}, \ref{lem.2.2} and \ref{cor.2.0}.

{\rm (ii)} The proof of the statement follows by using Lemmas
\ref{lem3.8} - \ref{lem3.900}, \ref{cor.2.0} and Theorem
\ref{thm.2.1}.
\end{proof}

\

\noindent{\bf The case $0 < p < \i$, $q = \i$.} The following
lemma is true.
\begin{lem}\label{lem3.10}
Let $0 < p<\i$ and let $u,\,v,\,w$ be weights. Assume that $u$ is
such that $U$ is admissible. Let $\vp$, defined by \eqref{eq.2.0},
be non-degenerate with respect to $U^{\frac{1}{p}}$. Let $\{x_k\}$
be any discretizing sequence for $\vp$.

{\rm (i)} If $0 < p < 1$, then
\begin{align*}
D  \ap &  \left\| \left\{ \vp(x_k) \left(\int_{[0,\i)} {\mathcal
U}(t,x_k)^{\frac{p^*}{p}}d \left( -
\underline{v}(t-,\i)^{p^*}\right) \right)^{\frac{1}{p^*}}\right\}
\right\|_{\ell^{\i}(\Z)}.
\end{align*}

{\rm (ii)} If $1 \le p < \i$, then
\begin{align*}
D  \ap &  \left\| \left\{ \vp(x_k) \left( \sup_{t \in (0,\i)}
{\mathcal U}(t,x_k)^{\frac{1}{p}} \underline{v}(t,\i)\right)\right\}
\right\|_{\ell^{\i}(\Z)}.
\end{align*}
\end{lem}
\begin{proof}
{\rm (i)} The proof of the statement follows by using Lemmas
\ref{lem1435786143515}, \ref{lem.2.800} and \ref{lem3.3} -
\ref{lem3.60}.

{\rm (ii)} The proof of the statement follows by combining Lemmas
\ref{lem1435786143515}, \ref{lem.2.800} and \ref{lem3.8} -
\ref{lem3.900}.
\end{proof}

Now we are in position to formulate our last main result.
\begin{thm}\label{main2002}
Let $0<p<\i$ and let $u,\,v,\,w$ be weights. Assume that $u$ is
such that $U$ is admissible. Let $\vp$, defined by \eqref{eq.2.0},
be non-degenerate with respect to $U^{\frac{1}{p}}$.

{\rm (i)} Let $0<p<1$. Then inequality \eqref{eq.4.2} holds for
every $h \in \mp^+(0,\i)$ if and only if
\begin{gather*}
I_5:= \sup_{x\in\I}\left(\esup_{s\in\I}w(s){\mathcal
U}(x,s)^{\frac{1}{p}} \right)\left(\int_{[0,\i)}{\mathcal
U}(t,x)^{\frac{p^*}{p}} d
(-\underline{v}(t-,\i)^{p^*})\right)^{\frac{1}{p^*}}<\i.
\end{gather*}
Moreover, the best constant $c$ in \eqref{eq.4.2} satisfies that
$c\ap I_5$.

{\rm (ii)} Let $1\leq p<\i$. Then inequality \eqref{eq.4.2} holds
for every $h \in \mp^+(0,\i)$ if and only if
\begin{gather*}
I_6 : = \sup_{x\in\I} \left(\esup_{s\in\I}w(s){\mathcal
U}(x,s)^{\frac{1}{p}} \right) U(x)^{-\frac{1}{p}}\sup_{t\in
(0,x)}U(t)^{\frac{1}{p}} \underline{v}(t,\i)<\i.
\end{gather*}
Moreover, the best constant $c$ in \eqref{eq.4.2} satisfies that
$c\ap I_6$.

\end{thm}
\begin{proof}
Both statements of the theorem follow by using Lemmas
\ref{lem1435786143515}, \ref{lem3.10}, \ref{lem.2.2} and
\ref{cor.2.0}.
\end{proof}

\section{Some Applications}\label{Some.Applications}

In this Section we give some applications of the obtained results. We
start with the weighted Hardy inequality on the cone of
non-increasing functions. Denote by $H_u$ the weighted Hardy
operator
$$
H_u f (x) : = \frac{1}{U(x)}\int_0^x f(t) u(t)\,dt, \qq x\in \I
$$
Note that the characterization of the weighted Hardy inequality on
the cone of non-increasing functions
\begin{equation}\label{eq.main3}
\|H_u f\|_{q,w,\I} \le c \|f\|_{p,v,\I}, \qq f \in
\mp^+(0,\i;\downarrow).
\end{equation}
has been obtained in \cite{bg} and \cite{gjop}.

The following reduction theorem is true.
\begin{thm}\label{Reduction.thm.1.1}
Let $0 < p,\,q  < \i$, and let $u,\,v,\,w$ be weights. Then the
inequality \eqref{eq.main3} holds for every $f \in
\mp^+(0,\i;\downarrow)$ if and only if the inequality
\begin{equation}\label{eq.main3021}
\left( \int_0^{\i} \left( \frac{1}{U(x)}\int_0^x \left(
\int_t^{\i}h \right)^{\frac{1}{p}} u(t)\,dt \right)^q w(x)\,dx
\right)^{\frac{p}{q}} \le C \int_0^{\i} h(t) V(t)\,dt
\end{equation}
holds for all $h\in\mp^+\I$. Moreover, the best constants $c$ and
$C$ in \eqref{eq.main3} and \eqref{eq.main3021}, respectively,
satisfy $C \ap c^p$.
\end{thm}
\begin{proof}
It is well-known that every non-negative, non-increasing function
$f$ is the pointwise limit of an increasing sequence of functions
of the form $\int_s^{\i} h$ for $h \ge 0$ (cf. \cite{ss}, p. 97).
Since $f$ is non-increasing if and only if $f^p$ is
non-increasing, by the Monotone Convergence Theorem,
\eqref{eq.main3} is equivalent to
\begin{equation*}\label{eq.main301}
\left( \int_0^{\i} \left( \frac{1}{U(x)}\int_0^x \left(
\int_t^{\i}h \right)^{\frac{1}{p}} u(t)\,dt \right)^q w(x)\,dx
\right)^{\frac{p}{q}} \le c^p \int_0^{\i} \left( \int_t^{\i}h
\right) v(t)\,dt,\,\, h\in\mp^+\I,
\end{equation*}
which, by Fubini's Theorem, is equivalent to
\begin{equation*}\label{eq.main302}
\left( \int_0^{\i} \left( \frac{1}{U(x)}\int_0^x \left(
\int_t^{\i}h \right)^{\frac{1}{p}} u(t)\,dt \right)^q w(x)\,dx
\right)^{\frac{p}{q}} \le c^p \int_0^{\i} h(t) V(t)\,dt,\,\,
h\in\mp^+\I.
\end{equation*}
\end{proof}
Analogously the following theorem can be proved:
\begin{thm}\label{Reduction.thm.1.2}
Let $0 < p  < \i$, and let $u,\,v,\,w$ be weights. Then the
inequality
\begin{equation}\label{eq.main3000000}
\|H_u f\|_{\i,w,\I} \le c \|f\|_{p,v,\I}
\end{equation}
holds for every $f \in \mp^+(0,\i;\downarrow)$ if and only if the
inequality
\begin{equation}\label{eq.main3021000000}
\esup_{x\in\I} w(x)^p \left( \frac{1}{U(x)}\int_0^x \left(
\int_t^{\i}h \right)^{\frac{1}{p}} u(t)\,dt \right)^p \le C
\int_0^{\i} h(t) V(t)\,dt
\end{equation}
holds for all $h\in\mp^+\I$. Moreover, for the best constants $c$ and
$C$ in \eqref{eq.main3000000} and \eqref{eq.main3021000000},
respectively, it yields that $C \ap c^p$.
\end{thm}

Combining Theorem \ref{Reduction.thm.1.1} with Theorems
\ref{main1001} and \ref{main1002}, we obtain the following
statement.
\begin{thm}\label{main3}
Let $u,\,v,\,w$ be weights. Assume that $u$ is such that $U$ is
admissible and the measure $w(t)dt$ is non-degenerate with respect
to $U^q$.

{\rm (i)} Let $0 < p \le 1$, $p\leq q <\i$. Then the inequality
\eqref{eq.main3}  holds for every $f \in \mp^+(0,\i;\downarrow)$
if and only if
$$
C_1:=\sup_{x \in \I}\left(\int_0^{\i}{\mathcal U}(x,t)^q
w(t)dt\right)^{\frac{1}{q}} {U(x)}^{-1}\left(\sup_{t\in (0,x)}U(t)
{V(t)}^{-\frac{1}{p}}\right)<\i.
$$
Moreover, the best constant $c$ in \eqref{eq.main3} satisfies that
$c\ap C_1$.

{\rm (ii)} Let $0 < p \le 1$, $0 < q < p$. Then the inequality
\eqref{eq.main3} holds for every $f \in \mp^+(0,\i;\downarrow)$ if
and only if
$$
C_2:=\left(\int_0^{\i} \left(\int_0^{\i}{\mathcal
U}(x,t)^qw(t)dt\right)^{\frac{q}{p-q}} U(x)^{\frac{pq}{q-p}}
\left(\sup_{t\in(0,x)}U(t)^{\frac{pq}{p-q}}{V(t)}^{\frac{q}{q-p}}\right)w(x)\,dx\right)^{\frac{p-q}{pq}}<\i.
$$
Moreover, the best constant $c$ in \eqref{eq.main3} satisfies that
$c\ap C_2$.

{\rm (iii)} Let $1 < p \le q < \i$. Then the inequality \eqref{eq.main3}
holds for every $f \in \mp^+(0,\i;\downarrow)$ if and only if
$$
C_3:=\sup_{x\in\I}\left(\int_0^{\i}{\mathcal U}(x,t)^q
w(t)dt\right)^{\frac{1}{q}} \left(\int_0^{\i} {\mathcal
U}(t,x)^{p'}\frac{v(t)}{V(t)^{p'}}\,dt \right)^{\frac{1}{p'}}<\i.
$$
Moreover, the best constant $c$ in \eqref{eq.main3} satisfies that
$c\ap C_3$.

{\rm (iv)} Let $1 < p < \i$, $0 < q < p$. Then the inequality
\eqref{eq.main3} holds for every $f \in \mp^+(0,\i;\downarrow)$ if
and only if
$$
C_4:=\left(\int_0^{\i} \left(\int_0^{\i}{\mathcal U}(x,t)^q
w(t)dt\right)^{\frac{q}{p-q}} \left( \int_0^{\i} {\mathcal
U}(t,x)^{p'}\frac{v(t)}{V(t)^{p'}}\,dt
\right)^{\frac{q(p-1)}{p-q}}w(x)\,dx\right)^{\frac{p-q}{pq}}<\i.
$$
Moreover, the best constant $c$ in the \eqref{eq.main3} satisfies
that $c\ap C_4$.
\end{thm}

Combining Theorems \ref{Reduction.thm.1.2} and \ref{main2002}, we
arrive at the following statement.
\begin{thm}\label{main30000}
Let $u,\,v,\,w$ be weights. Assume that $u$ is such that $U$ is
admissible. Let $\vp$, defined by
$$
\vp(t) : = \esup_{s\in (0,t)} U(s) \esup_{\tau \in
(s,\i)}\frac{w(\tau)}{U(\tau)},\qq t\in\I,
$$
be non-degenerate with respect to $U$.

{\rm (i)} Let $0 < p \le 1$. Then the inequality
\eqref{eq.main3000000} holds for every $f \in
\mp^+(0,\i;\downarrow)$ if and only if
$$
C_5:=\sup_{x\in\I}\left(\esup_{s\in\I}w(s){\mathcal U}(x,s)\right)
{U(x)}^{-1}\left(\sup_{t\in (0,x)}U(t)
{V(t)}^{-\frac{1}{p}}\right)<\i.
$$
Moreover, the best constant $c$ in \eqref{eq.main3000000}
satisfies that $c\ap C_5$.

{\rm (ii)} Let $1 < p < \i$. Then the inequality
\eqref{eq.main3000000} holds for every $f \in
\mp^+(0,\i;\downarrow)$ if and only if
$$
C_6:=\sup_{x\in\I}\left(\esup_{s\in\I}w(s){\mathcal U}(x,s)\right)
\left(\int_0^{\i} {\mathcal U}(t,x)^{p'}\frac{v(t)}{V(t)^{p'}}\,dt
\right)^{\frac{1}{p'}}<\i.
$$
Moreover, the best constant $c$ in \eqref{eq.main3000000}
satisfies that $c\ap C_6$.
\end{thm}

Now we consider the generalized Stieltjes transform $S$ defined by
$$
(Sh)(x) = \int_0^{\i} \frac{h(t)\,dt}{U(x) + U(t)}
$$
for all $h \in \mp^+\I$; the usual Stieltjes transform is obtained
on putting $U(x) \equiv x$. In the case $U(x) \equiv x^{\la}$,
$\la > 0$, the boundedness of the operator $S$ between weighted
$L^p$ and $L^q$ spaces was investigated in \cite{a} (when $1\le p
\le q \le \i$) and in \cite{s2}, \cite{s1} (when $1\le q < p\le
\i$). This problem also was considered in \cite{GKP} and
\cite{GPSW}, where completely different approach was used, based
on the so call ``gluing lemma'' (see also \cite{gkuf.per.}).

The following reduction theorem is true.
\begin{thm}\label{Reduction.thm.1.3}
Let $0 < q  \le \i$, $1 \le p \le \i$, and let $u,\,v,\,w$ be
weights. Then the inequality
\begin{equation}\label{Stieltjes.ineq.01}
\|Sh\|_{q,w,\I} \le c \|h\|_{p,v,\I}, \qq h \in \mp^+\I
\end{equation}
holds if and only if
\begin{equation}\label{Stieltjes.ineq.0100}
\left\|H_u \left( \int_t^{\i} h\right)\right\|_{q,w,\I} \le c
\|hU\|_{p,v,\I},\,\, h\in \mp^+\I
\end{equation}
holds.
\end{thm}
\begin{proof}
Evidently, inequality \eqref{Stieltjes.ineq.01} is equivalent to
the following inequality:
$$
\|S(hU)\|_{q,w,\I} \le c \|hU\|_{p,v,\I}, \qq h \in \mp^+\I.
$$
It is easy to see that
$$
S(hU)(x) \ap \frac{1}{U(x)}\int_0^x \left( \int_t^{\i}
h(s)\,ds\right)u(t)\,dt, \qq h\in \mp^+\I.
$$
Indeed, by Fubini's Theorem, we have that
\begin{align*}
\int_0^x \left( \int_t^{\i} h(s)\,ds\right)u(t)\,dt & = \int_0^x
\left( \int_t^x h(s)\,ds + \int_x^{\i} h(s)\,ds\right)u(t)\,dt \\
& = \int_0^x \int_0^s u(t)\,dt h(s)\,ds  + \int_x^{\i} h(s)\,ds
\int_0^x u(t)\,dt \\
& = \int_0^x U(s)h(s)\,ds  + U(x) \int_x^{\i} h(s)\,ds \\
& \ap U(x) \int_0^{\i}\frac{U(s)}{U(x) + U(s)}h(s)\,ds =
U(x)S(hU)(x),
\end{align*}
that is,
$$
S(hU)(x) \ap H_u \left( \int_t^{\i} h\right)(x), \qq x \in \I.
$$
Hence, we see that the inequality \eqref{Stieltjes.ineq.01}  is
equivalent to the inequality \eqref{Stieltjes.ineq.0100}.
\end{proof}

Combining Theorem \ref{Reduction.thm.1.3} with  Theorems
\ref{main1002}, \ref{main2002} and Theorem 3.1, 3.2 in \cite{gmp},
we obtain the following statements.
\begin{thm}\label{Stieltjes.ineq.thm.1.1}
Let $u,\,v,\,w$ be weights. Assume that $u$ is such that $U$ is
admissible and the measure $w(t)dt$ is non-degenerate with respect
to $U^q$. Let $p,\,q \in (0,\i]$. When $q < p < \i$, we set $r =
\frac{pq}{p-q}$.

{\rm (i)} Let $p=1$, $1\leq q <\i$. Then the  inequality
\eqref{Stieltjes.ineq.01} holds for every $h \in \mp^+(0,\i)$ if
and only if
$$
S_1:=\sup_{x \in \I}\left(\int_0^{\i}{\mathcal U}(x,t)^q
w(t)dt\right)^{\frac{1}{q}} {U(x)}^{-1}\sup_{t\in (0,x)}{U(t)}
\esup_{s\in (t,\i)} (U(s)v(s))^{-1} < \i.
$$
Moreover, the best constant $c$ in \eqref{Stieltjes.ineq.01}
satisfies that $c\ap S_1$.

{\rm (ii)} Let $p=1$, $0<q<1$. Then the inequality
\eqref{Stieltjes.ineq.01} holds for every $h \in \mp^+(0,\i)$ if
and only if
\begin{align*}
S_2:=\left(\int_0^{\i} \left(\int_0^{\i}{\mathcal
U}(x,t)^qw(t)dt\right)^{q^*} U(x)^{-{q^*}} \right.
& \\
& \hspace{-4cm}\left. \times
\left(\sup_{t\in(0,x)}U(t)^{{q^*}}{\esup_{s\in
(t,\i)}(U(s)v(s))^{-{q^*}}}\right){w(x)}\,dx\right)^{\frac{1}{q^*}}<\i.
\end{align*}
Moreover, the best constant $c$ in \eqref{Stieltjes.ineq.01}
satisfies that $c\ap S_2$.

{\rm (iii)} Let $1 < p \le q < \i$. Then the inequality
\eqref{Stieltjes.ineq.01} holds for every $h \in \mp^+(0,\i)$ if
and only if
\begin{align*}
S_3:=\sup_{x\in\I}\left(\int_0^{\i}{\mathcal U}(x,t)^q
w(t)dt\right)^{\frac{1}{q}} \left( \int_0^{\i} {\mathcal
U}(t,x)^{p'} U(t)^{-p'}v(t)^{1-p'}\,dt\right)^{\frac{1}{p'}} < \i.
\end{align*}
Moreover, the best constant $c$ in \eqref{Stieltjes.ineq.01}
satisfies that $c\ap S_3$.

{\rm (iv)} Let $1 < p  < \i$, $0 < q < p$. Then the inequality
\eqref{Stieltjes.ineq.01} holds for every $h \in \mp^+(0,\i)$ if
and only if
\begin{align*}
S_4:= \left(\int_0^{\i} \left(\int_0^{\i}{\mathcal U}(x,t)^q
w(t)dt\right)^{\frac{r}{p}} \left( \int_0^{\i}{\mathcal
U}(t,x)^{p'}
U(t)^{-p'}v(t)^{1-p'}\,dt\right)^{\frac{r}{p'}}w(x)\,dx
\right)^{\frac{1}{r}} < \i.
\end{align*}
Moreover, the best constant $c$ in \eqref{Stieltjes.ineq.01}
satisfies that $c\ap S_4$.

{\rm (v)} Let $p = \i$, $0 <  q < \i$. Then the inequality
\eqref{Stieltjes.ineq.01} holds for every $h \in \mp^+(0,\i)$ if
and only if
\begin{align*}
S_5:= \left(\int_0^{\i} \left({\mathcal U}(t,x) U(t)^{-1}
\frac{dt}{v(t)} \right)^q w(x)\,dx \right)^{\frac{1}{q}} < \i.
\end{align*}
Moreover, the best constant $c$ in \eqref{Stieltjes.ineq.01}
satisfies that $c\ap S_5$.
\end{thm}

\begin{thm}\label{Stieltjes.ineq.thm.1.2}
Let $u,\,v,\,w$ be weights. Assume that $u$ is such that $U$ is
admissible. Let $\vp$, defined by
$$
\vp(t) : = \esup_{s\in (0,t)} U(s) \esup_{\tau \in
(s,\i)}\frac{w(\tau)}{U(\tau)},\qq t\in\I,
$$
be non-degenerate with respect to $U$.

{\rm (i)} Let $p=1$. Then the  inequality
\begin{equation}\label{Stieltjes.ineq.02}
\|Sh\|_{\i,w,\I} \le c \|h\|_{p,v,\I}
\end{equation}
holds for every $h \in \mp^+(0,\i)$ if and only if
\begin{gather*}
S_6 : = \sup_{x\in\I} \left(\esup_{s\in\I}w(s){\mathcal U}(x,s)
\right) U(x)^{-1}\sup_{t\in (0,x)}U(t) \esup_{s\in (t,\i)}
(U(s)v(s))^{-1}<\i.
\end{gather*}
Moreover, the best constant $c$ in \eqref{Stieltjes.ineq.02}
satisfies that $c\ap S_6$.

{\rm (ii)} Let $1 < p < \i$. Then inequality
\eqref{Stieltjes.ineq.02} holds for every $h \in \mp^+(0,\i)$ if
and only if
\begin{gather*}
S_7 : = \sup_{x\in\I} \left(\esup_{s\in\I}w(s){\mathcal U}(x,s)
\right) \left( \int_0^{\i}{\mathcal U}(t,x)^{p'}
U(t)^{-p'}v(t)^{1-p'}\,dt\right)^{\frac{1}{p'}}<\i.
\end{gather*}
Moreover, the best constant $c$ in \eqref{Stieltjes.ineq.02}
satisfies that $c\ap S_7$.

{\rm (iii)} Let $p = \i$. Then inequality
\eqref{Stieltjes.ineq.02} holds for every $h \in \mp^+(0,\i)$ if
and only if
\begin{gather*}
S_8 : = \sup_{x\in\I} \left(\esup_{s\in\I}w(s){\mathcal U}(x,s)
\right) \left( \int_0^{\i}{\mathcal U}(t,x) U(t)^{-1}
\,\frac{dt}{v(t)}\right)<\i.
\end{gather*}
Moreover, the best constant $c$ in \eqref{Stieltjes.ineq.02}
satisfies that $c\ap S_8$.
\end{thm}

\vskip0.4cm
\begin{flushleft}
Amiran Gogatishvili\\
Institute of Mathematics
\\Academy of Sciences of Czech Republic  \\
Zitna 25, CZ - 115 67 Praha 1,\\
Czech Republic \\
\vskip0.2cm E-mail: gogatish@math.cas.cz
\end{flushleft}

\vskip0.2cm
\begin{flushleft}
Rza Mustafayev \\
Department of Mathematics, \\
Faculty of Science and Arts, \\
Kirikkale University, 71450 Yahsihan, Kirikkale, \\
Turkey \\
\vskip0.2cm E-mail: rzamustafayev@gmail.com
\end{flushleft}

\vskip0.2cm
\begin{flushleft}
Lars-Erik Persson\\

Department of Engineering Sciences and Mathematics, \\
Lule\aa\ University of Technology,\\
SE-971 87 Lule\aa , Sweden. \\

\vskip0.2cm
Narvik University College,\\
P.O. Box 385, N 8505 Narvik, Norway \\
\vskip0.2cm E-mail: larserik@ltu.se
\end{flushleft}
\end{document}